\documentclass[11pt]{amsart}
\usepackage{amssymb}
\usepackage{amsmath}
\usepackage{graphicx}
\usepackage{float}
\usepackage{amsthm}
\usepackage[breaklinks]{hyperref}
\newtheorem{theorem}{Theorem}
\newtheorem{corollary}{Corollary}
\newtheorem{proposition}{Proposition}
\newtheorem{example}{Example}

\newtheorem{remark}{Remark}

\title[Radiotherapy fractionation problem]{Analytical solution to the radiotherapy fractionation problem including \\ dose bound constraints}
\author{Luis A. Fern\'andez and Luc\'{\i}a Fern\'andez}
\dedicatory{Dedicated to the memory of Juan Antonio Fern\'andez (1957-2018)}
\address{Dep. Mathematics, Statistics and Computation \\ University of Cantabria (SPAIN)}
\email{lafernandez@unican.es}
\thanks{The work of the first author was supported by the Spanish ``Ministerio de Econom\'{\i}a, Industria y Competitividad" under project  MTM2017-83185-P}

\begin{document}

\begin{abstract}
This paper deals with the classic radiotherapy dose fractionation problem for cancer tumors concerning the following goals:
\begin{itemize}
\item [a)] To maximize the effect of radiation on the tumor, restricting the effect produced to the organs at risk (healing approach).
\item [b)] To minimize the effect of radiation on the organs at risk, while maintaining enough effect of radiation on the tumor (palliative approach).
\end{itemize}
We will assume the linear-quadratic model to characterize the radiation effect and consider the stationary case (that is, without taking into account
the timing of doses and the tumor growth between them).

The main novelty with respect to previous works concerns the presence of minimum and maximum dose fractions,
to achieve the minimum effect and to avoid undesirable side effects, respectively.

We have characterized in which situations is more convenient the hypofractionated protocol (deliver few fractions with high dose per
fraction) and in which ones the hyperfractionated regimen (deliver a large number of lower doses of radiation) is the optimal strategy.

In all cases, analytical solutions to the problem are obtained in terms of the data.
In addition, the calculations to implement these solutions are elementary and can be carried out using a pocket calculator.
\end{abstract}
\maketitle
\section{Introduction}
According to the World Health Organization \cite{WHO}, ``radiotherapy is one of the major treatment options in cancer management. (...)  Together with other modalities such as
surgery and chemotherapy it plays an important role in the treatment of $40\%$ of those patients who are cured of their cancer. Radiotherapy is also a highly effective
treatment option for palliation and symptom control in cases of advanced or recurrent cancer. The process of radiotherapy is complex and involves understanding of the
principles of medical physics, radiobiology, radiation safety, dosimetry, radiotherapy planning, simulation and interaction of radiation therapy with other treatment
modalities".

Mathematical modelling has played an important role in understanding and optimizing radiation delivery for cancer treatment.
Since its formulation more than $50$ years ago, the linear-quadratic (LQ) model has become the preferred method for characterizing radiation effects.
Usually, it is stated as follows: the survival probability $S$ of a tumor cell after exposure to a single dose of radiation of $d \ Gy$ is expressed as
$$ S = \exp{\left(-\alpha_{T} d-\beta_{T} d^{2}\right)}, $$
where $\alpha_{T}$ and $\beta_{T}$ are two positive parameters describing the radiosensitivity of the cell, \cite{ref3}.
It is well known that these parameters depend on the type of radiation therapy chosen and also on the organ where the tumor is located \cite{ref4}.
More precisely, LQ model implies that if the initial size of the tumor is $U$, then it will be $U \cdot S$ after applying a $d \ Gy$ dose.
Let us recall that ``Gray" ($Gy$) is the unit of ionizing radiation dose in the International System of Units.

LQ model has well documented predictive properties for fractionation/dose rate effects in the laboratory and
``it is reasonably well validated, experimentally and theoretically, up to about $10 \ Gy$  per fraction and would be reasonable for use up
to about $18 \ Gy$ per fraction", see \cite{Brenner2008}. Precisely, its range of validity is a key point of controversy; although
there is a general consensus on the existence of this range, significant disagreements
remain on the exact values of its limits. Let us illustrate this fact with other recent quotes:
from  \cite{ref3}, ``in vitro (...) some authors suggesting significant discrepancies at doses of $5 \ Gy$ or above, while others
report good agreement up to tens of Gy" and according to the French Society of Young Radiation Oncologists,  ``the dose / fraction must be between $1$ and $6 \ Gy$", see \cite{ref8}.

Given $N$ doses, $d_{1},...,d_{N}$, eventually different, if we consider the stationary case (which means that neither the times of application of the doses
nor the growth of the tumor produced between them are taken into account), the probability of accumulated survival is given by
\begin{equation}
S^{N}=\exp{\left(-\sum_{i=1}^{N} \left(\alpha_{T} d_{i} + \beta_{T} d_{i}^{2}\right)\right)}.
\label{ec1}
\end{equation}

From here it is clear that the effect of radiation on the tumor is determined by the quadratic function
\begin{equation}
E_{T}(N,d) = \alpha_{T}\sum_{i=1}^{N} d_{i} +\beta_{T} \sum_{i=1}^{N} d_{i}^{2}.
\label{ET}
\end{equation}

On the other hand, radiation also affects healthy organs and tissues near the tumor (which we will denote by OAR, organs at risk, hereafter).
In general, healthy organs and tissues receive less radiation than the tumor, which we will denote by $\delta d$, with $\delta \in (0,1]$ being the so-called ``sparing factor".
The value of $\delta$ depends on factors such as the location and geometry of the tumor and also on
the technology used to deliver the radiation, see \cite{Bortfeld2015}. It can be seen as a measure of the accuracy of the radiotherapy:
if clinicians can keep the OAR almost unaffected by the radiation, $\delta$ will be about $0$; if not,
it will be larger, until reaching the value $\delta \approx 1$ at worst.
Therefore, the effect of the radiation on the OAR is determined by the following function

\begin{equation}
E_{OAR}(N,d) = \alpha_{0}\delta \sum_{i=1}^{N} d_{i} + \beta_{0}\delta^2  \sum_{i=1}^{N} d_{i}^{2},
\label{EOAR}
\end{equation}
where $\alpha_{0}$ and $\beta_{0}$ are the parameters associated to the healthy organs that we are trying to protect.

Typical values for $\alpha_{0}, \beta_{0}, \alpha_{T}$ and $\beta_{T}$ can be found in the specialized literature such as \cite{ref4}.
These data come from conducting experiments and the corresponding adjustments
(least squares regression) to achieve approximated values that best fit experimental data.

Let us now introduce the most common strategies for fractionating radiotherapy treatments:
\begin{itemize}
\item {\bf Hypofractionation}: Deliver higher doses of radiation on few occasions.
This strategy results in a significant reduction of its duration.
\item {\bf Hyperfractionation}: Deliver a large number of lower doses of radiation
that are given more than once a day.
\end{itemize}

In this paper we study the classic radiotherapy dose fractionation problem related to the following goals:
\begin{itemize}
\item [a)] To maximize the effect of radiation on the tumor, restricting the effect produced on the OAR (healing approach) in Section 2 and
\item [b)] To minimize the effect of radiation on the OAR, maintaining enough effect of radiation on the tumor (palliative approach) in Section 3.
\end{itemize}

The first novelty with respect to previous works in this framework concerns the presence of dose fraction bounds of the type $0 < d_{min} \leq d \leq d_{max}$.
On one hand, these restrictions are connected to the range of validity of the aforementioned LQ model and can be estimated  for each particular tumor; on the other hand, they also take into account the minimum and maximum dose fraction that can be applied in practical situations in order to achieve a minimum effect and avoid undesirable side effects, respectively.
It is well known that the dose per fraction value in most conventional treatments is around $2  \ Gy,$ see for instance \cite{RDF2019}. Depending on the tumor type, the values of  $d_{min}$ and $d_{max}$
can be tunned, but the reference values  $d_{min} = 1 \ Gy$ and $d_{max} = 6 \ Gy$ could be a valid generic choice. In this sense one can not find in \cite{RDF2019} a single treatment recomendation with a dose fraction less than $1 \ Gy$ and very few larger than $6 \ Gy.$

The counterpart for imposing a positive minimum dose fraction is that the total number of radiations $N$ should not be fixed a priori and this is the second important novelty of this work: $N$ will also be considered another unknown of the problem and we will study the dependence of the solution with respect to $N$. Among others, this approach was followed by \cite{Jones1995}, but only for uniform dose treatments. Our approach here includes also the study for nonuniform protocols. A preliminary version of our results was presented by the second author as part of the academic project \cite{TFG}, except for the study of the dependence of the solution with respect to $N$ which is new.

Summarizing, new analytical solutions in terms of the data are obtained for both problems, improving known results in the literature to the best of our knowledge, see for instance \cite{Mizuta2012} and \cite{Saberian2015}. Moreover, the final calculations to implement these solutions are elementary and can be carried out using a pocket calculator.

\section{Maximizing the effect of radiation on the tumor}
\label{Smax}

The aim of this first problem is to determine the best strategy to maximize the effect of radiation on the tumor, while restricting the effect on the OAR (healing approach):
$$ (P_1) \left\{ \begin{array}{ll} \begin{aligned}
& \textrm{Maximize } E_{T}(N,d),  \\
& \textrm{subject to } N \in \mathbb{N}, d \in \mathbb{R}^{N} \textrm{ such that  } \\
& E_{OAR}(N,d) \leq \gamma_{OAR}, \\
& d_{min}\leq d_{i}\leq d_{max}, \ i=1,...,N,
\end{aligned}
\end{array}
\right.$$
where $E_{T}(N,d)$ is given by (\ref{ET}), $E_{OAR}(N,d)$ by (\ref{EOAR}) and $d_{min}, d_{max}$ and $\gamma_{OAR}$ are a priori known positive parameters, that should be provided by the specialists.
Roughly speaking, the restriction $E_{OAR}(N,d) \leq \gamma_{OAR}$ can be interpreted in the sense that the percentage of survival cells of the OAR should be greater than or equal to
$100 \exp{(-\gamma_{OAR})}.$

This is the classic fractionation problem that has been studied (with some variations) in several works, see for example the recent papers \cite{Bertuzzi2013} and \cite{Saberian2015} (where more than one OAR is considered) and the references therein. The first novelty of our approach is that dose bound constraints are also included. Usually in the literature the lower bound $0$ value is taken for $d_i$ and no upper bound is imposed; some exceptions are \cite{Bruni2015} and \cite{Bruni2019} where an upper bound is included, but not a positive lower bound. The danger of losing control of the tumor, due to the use of doses below a critical limit, has already been pointed out by \cite{Jones1995}. In addition, our approach to the problem is more useful since the number of doses $N$ is not initially set as in  \cite{Bertuzzi2013} and \cite{Saberian2015}. The case including repopulation was studied in \cite{Bortfeld2015}, only assuming the non-negativity of $d_i$.

From a mathematical point of view, this is a mixed optimization problem involving a discrete variable,  $N \in \mathbb{N}$, which corresponds to the number of radiation doses, and $N$ continuous variables, $d_{i} \in \mathbb{R}, 1 \leq i \leq N$, which are the doses. In other words, this problem has the peculiarity of having a variable number of unknowns.

Along this paper, it will be denoted
\begin{equation}
\varphi_{0}(r)=\alpha_{0} \delta r + \beta_{0} \delta^{2} r^{2},
\label{var0}
\end{equation}
\begin{equation}
\lambda_{0} = max\left\lbrace  1, \dfrac{\gamma_{OAR}}{\varphi_{0}(d_{max})} \right\rbrace,
\label{lam0}
\end{equation}
and
\begin{equation}
\rho_{0} = \dfrac{\gamma_{OAR}}{\varphi_{0}(d_{min})}.
\label{rho0}
\end{equation}

Also,  we will denote by $\lfloor x \rfloor$ the greatest integer less than or equal to $x$ and by $\lceil x \rceil$ the least integer greater than or equal to $x$.
Finally, the notation $d^N = (d_0,\ldots,d_0)$ means that $d^N \in \mathbb{R}^{N}$ having all the $N$ components equal to $d_0$.

\subsection{Existence of solution for $(P_1)$}
\begin{theorem}
Let us assume $d_{min}>0$ and $\rho_{0} \geq 1$. Then, the problem $(P_{1})$ has (at least) one solution.
\label{teo1}
\end{theorem}

\begin{proof}
Taking into account the restrictions for $(P_{1})$ and that $d_{min}>0$, we have $N \leq \rho_{0}.$
Hence, the set of feasible values for $N$ is finite.

If $\rho_0 = 1$, the solution is $(N,d) = (1,d_{min})$, because this is the only admissible pair  for $(P_{1})$.

When $\rho_0 \in (1,2),$ the value $N =1$ is still the only possible option. Consequently, we are faced with a maximizing problem of an increasing $1D$ function.
Then, the solution will be given by the largest feasible value. In this case, it is quite easy to verify that the unique solution of $(P_1)$ is the pair
$(1,\min{\{d_{max},\overline{d}_{0}\}}),$ where $\overline{d}_{0}=\dfrac{-\alpha_{0} + \sqrt{\alpha^2_{0}+4\beta_{0} \gamma_{OAR}}}{2\beta_{0} \delta}$. Let us stress that $\varphi_0(\overline{d}_{0}) = \gamma_{OAR}$.

If $\rho_0 \geq 2$ we can reduce the problem $(P_{1})$ to a finite collection
of continuous optimization problems $(P_{1}^{N})$ with fixed $N$ given by:
$$(P_{1}^{N}) \left\{
\begin{array}{ll}
\begin{aligned}
& \textrm{Maximize} \hspace{0.4cm} \tilde{E}^N_{T}(d) = \alpha_{T} \sum_{i=1}^{N} d_{i} +\beta_{T} \sum_{i=1}^{N}d_{i}^{2}, \\
& \textrm{subject to } d \in \mathbb{R}^{N} \textrm{ such that }\\
& E_{OAR}(N,d) \leq \gamma_{OAR}, \\
& d_{min}\leq d_{i}\leq d_{max}, \ \ \ \ i=1,...,N.
\end{aligned}
\end{array}
\right. $$
Firstly we will prove the existence of a solution for each problem $(P_{1}^{N})$ (see Theorem \ref{exist}), for $N$ running $\left[ 1,\rho_{0} \right] \cap \mathbb{N}$ and denote it by $\overline{d}^{N}$. Then, it is enough to take the pair  $\left( \overline{N},\overline{d}^{\overline{N}} \right) $ from the finite set  $$\left\lbrace \left( N,\overline{d}^{N}\right): N \in \left[1, \rho_{0} \right] \cap \mathbb{N} \right\rbrace, $$ that maximizes the value of $E_{T}(N,d)$  as a solution to the problem $(P_{1})$.
\end{proof}
The existence of a solution for each problem $(P_{1}^{N})$ is proved below:
\begin{theorem}
Let us assume $d_{min}>0$, $\rho_{0} \geq 2$ and $N \in \left[1, \rho_{0} \right] \cap \mathbb{N}$. Then
 the problem $(P_{1}^{N})$ has (at least) one solution.
\label{exist}
\end{theorem}
\begin{proof}
For small values of $N$, more precisely $N \in \left[1, \lambda_{0} \right] \cap \mathbb{N}$, it is easy to verify that the solution for $(P_{1}^{N})$ is the trivial one with maximum doses, that is, $ \overline{d}^{N}=(d_{max},...,d_{max}).$
For other values, $N \in \left( \lambda_{0}, \rho_{0} \right] \cap \mathbb{N}$, the existence of solution for $(P_{1}^{N})$ follows from the classic Weierstrass Theorem, because we are maximizing a continuous objective function over a compact set.
\end{proof}

\begin{remark}
\begin{itemize}
\item [a)] Let us point out that $(P_{1})$ is a nonconvex quadratically constrained quadratic optimization problem (even $(P_{1}^{N})$ with fixed $N$), because the objective is to maximize a convex function.
Typically, this type of problems is computationally difficult to solve (see \cite{Saberian2015}), but here we will see that it can be done analytically.
\item [b)] Unless all the components of the solution are the same, the uniqueness of solution fails: it is enough to take two indices $i,j \in \lbrace 1,...,\overline{N} \rbrace$ such that $\overline{d}_{i} \neq \overline{d}_{j}$ and interchange these coordinates to generate a new solution.
\item [c)] Under the condition $\rho_0 < 1$, it is apparent that the set of feasible points is empty and hence, the existence of solution for $(P_1)$ fails.
\item [d)] The hypothesis $d_{min} > 0$ is also needed for proving the existence of solution for $(P_1)$, as it can be shown through the following example:
\begin{example}
$$ (P_{10}) \left\{
\begin{array}{ll}
\textrm{Maximize  }  \displaystyle E_T(N,d)= \sum_{i=1}^{N} d_{i} + \sum_{i=1}^{N} d_{i}^{2}, \\
\textrm{subject to } N \in \mathbb{N}, d_{i} \in \mathbb{R}, \\
\displaystyle E_{OAR}(N,d) = \sum_{i=1}^{N} d_{i} + 2\sum_{i=1}^{N} d_{i}^{2} \leq 10, \\
0 \leq d_{i} \leq 1, i=1,...,N.
\end{array} \right.  $$
 \end{example}
It is clear that for all feasible points we have
\[ E_{T}(N,d)  \leq  E_{OAR}(N,d) \leq 10.\]
Let us stress that here $N$ can take any natural value, without restrictions.
Inspired by Theorem \ref{teo2} below, let us consider the sequence given by
\[ d^N = (d_{0N},...,d_{0N}), \ \ \mbox{ with } \ \ d_{0N} = \dfrac{-1}{4} + \sqrt{\dfrac{1}{16}+\dfrac{5}{N}}. \ \]
It is easy to check that it is feasible for $N \geq 4$,
\[ E_{OAR}(N,d^N) =10, \ \  E_{T}(N,d^N) = 10 - Nd_{0N}^2 \longrightarrow 10, \ \ \mbox{ as } N \rightarrow +\infty.\]
Hence, the problem $ (P_{10})$ can not have solution $(\overline{N},\overline{d})$: the supremum value $10$ can not be attained since
it should happen that
\[  10 + \sum_{i=1}^{\overline{N}} \overline{d}_{i}^2 = E_{T}(\overline{N},\overline{d}) + \sum_{i=1}^{\overline{N}}  \overline{d}_{i}^{2} = E_{OAR}(\overline{N},\overline{d}) \leq 10,\]
which is clearly impossible.
\item [e)] The deficiency of this type of models to produce solutions (when $d_{min} = 0$) that prescribe infinite doses with fractions tending to zero was pointed out by \cite{Jones1995}.
\end{itemize}
\label{R-1}
\end{remark}

The following result provides a simpler version of the optimization problem for the bigger values of $N$:
\begin{theorem}
Let us assume $d_{min}>0$, $\rho_{0} \geq 2$ and $N \in \left( \lambda_{0}, \rho_{0} \right] \cap \mathbb{N}$. Then, the  inequality constraint of the problem $(P_{1}^{N})$ has to be active at $\overline{d}^{N}$, with $\overline{d}^{N}$ being a solution for $(P_{1}^{N})$.
\label{teoAct}
\end{theorem}
\begin{proof}
Arguing by contradiction, let us assume that the constraint is not active, i.e.,
\begin{equation}
\alpha_{0} \delta \sum_{i=1}^{N} \overline{d}_{i} + \beta_{0} \delta^{2} \sum_{i=1}^{N} \overline{d}_{i}^{2} < \gamma_{OAR}.
\label{c1}
\end{equation}
Since $\lambda_{0} < N$, we know that there exists some index $j \in \lbrace 1,...,N \rbrace$ such that $\overline{d}_{j} < d_{max}$. Then, for sufficiently small $\epsilon >0$, the point $(\overline{d}_{1},..., \overline{d}_{j-1},\overline{d}_{j} +\epsilon,\overline{d}_{j+1},...,\overline{d}_{N})$ is feasible and satisfies
$$ \tilde{E}^N_{T}(\overline{d}^{N})<\tilde{E}^N_{T}((\overline{d}_{1},...,\overline{d}_{j-1},\overline{d}_{j} +\epsilon,\overline{d}_{j+1},...,\overline{d}_{N})),$$
but this contradicts the fact that $\overline{d}^{N}$ is a solution for $(P_{1}^{N})$.
\end{proof}
Hence, from now on, in this case we will consider the equality restriction
$$ \alpha_{0} \delta \sum_{i=1}^{N} d_{i} + \beta_{0} \delta^{2} \sum_{i=1}^{N} d_{i}^{2} = \gamma_{OAR}. $$
Therefore, we deduce that
$$ \sum_{i=1}^{N} d_{i}^{2}=\dfrac{1}{\beta_{0} \delta^{2}} \left[ \gamma_{OAR} - \alpha_{0} \delta \sum_{i=1}^{N} d_{i} \right],  $$
and the objective function can be written as
\begin{equation}
\tilde{E}^N_{T}(d)=\left[ \alpha_{T} - \dfrac{\beta_{T}\alpha_{0}}{\beta_{0} \delta} \right] \sum_{i=1}^{N} d_{i} + \dfrac{\beta_{T}\gamma_{OAR}}{\beta_{0}\delta^{2}}.
\label{eq1}
\end{equation}

Based on this identity, we can directly simplify the formulation of the problem $(P_{1}^{N})$ as follows:

 \begin{proposition}
Let us assume $d_{min}>0$, $\rho_{0} \geq 2$, $N \in \left( \lambda_{0}, \rho_{0} \right] \cap \mathbb{N}$ and denote
\begin{equation}
\displaystyle \omega_{\delta} = \frac{\alpha_{T}}{\beta_{T}} - \frac{\alpha_{0}}{\beta_{0} \delta}.
\label{Ew}
\end{equation}
\begin{itemize}
\item [i)] If $\omega_{\delta}>0$, then $(P_{1}^{N})$ is equivalent to
\begin{equation} (P_{1}^{N,+}) \left\{
\begin{array}{ll}
\textrm{Maximize  } \displaystyle \sum_{i=1}^{N} d_{i}, \\
\textrm{subject to } d \in \mathbb{K}_{1}^{N},
\end{array}
\right.
\label{E1}
\end{equation}
where
\begin{equation}
\mathbb{K}_{1}^{N} = \lbrace d \in \mathbb{R}^{N} : E_{OAR}(N,d) = \gamma_{OAR}, d_{min} \leq d_{i} \leq d_{max}, 1 \leq i \leq N \rbrace.
\label{E690}
\end{equation}
\item [ii)] If $\omega_{\delta}<0$, then $(P_{1}^{N})$ is equivalent to
\begin{equation} (P_{1}^{N,-}) \left\{
\begin{array}{ll}
\textrm{Minimize  } \displaystyle \sum_{i=1}^{N} d_{i}, \\
\textrm{subject to } d \in \mathbb{K}_{1}^{N}.
\end{array}
\right.
\label{E2}
\end{equation}
\item [iii)] If $\omega_{\delta}=0,$ then every feasible point for $(P_{1}^{N})$ is a solution.
\end{itemize}
\label{PROP1}
\end{proposition}

\begin{remark}
\begin{itemize}
\item [a)] The idea of this transformation can be found in \cite{Mizuta2012} in the context of the problem $(P_{2})$ that we will study in the next section.
\item [b)] Let us note that for the majority of tumors $\alpha_{T}/\beta_{T} > \alpha_{0}/\beta_{0}$ and therefore, the case $\omega_{\delta} > 0$ is more frequent in clinical practice.
\item [c)] As a consequence of Proposition \ref{PROP1}, we can appreciate the great difference between the cases $\omega_{\delta} > 0$ and $\omega_{\delta} < 0$:
in the first one, to maximize the effect of radiation on the tumor we have to increase the total dose, while in the second the total dose remains minimum (see also the Subsection \ref{EQTR} below).
\end{itemize}
\label{R1a}
\end{remark}

\subsection{Solving $(P_{1}^{N})$}
\label{SP1N}

Let us begin by showing a $2D$-example of previous problems that will inspire the general results of this section.
\begin{example}
Let us consider the following optimization problems:
$$ (P_{1}^{2,+}) \left\{
\begin{array}{ll}
\textrm{Maximize } \hspace{0.25cm} \displaystyle d_{1} + d_{2}, \\
\textrm{subject to } (d_{1},d_{2}) \in \mathbb{R}^{2}, \\
2(d_{1}+d_{2}) + d_{1}^{2}+d_{2}^{2}=12, \\
1 \leq d_{1},d_{2} \leq 3.
\end{array}
\right.
\hspace{0.3cm}
(P_{1}^{2,-}) \left\{ \begin{array}{ll}
\textrm{Minimize } \hspace{0.25cm} \displaystyle d_{1} + d_{2}, \\
\textrm{subject to } (d_{1},d_{2}) \in \mathbb{R}^{2}, \\
2 (d_{1}+d_{2}) + d_{1}^{2}+d_{2}^{2}=12, \\
1 \leq d_{1},d_{2} \leq 3.
\end{array}
\right.$$
In the Figure \ref{F1} the points on the blue surface are those that satisfy the equality constraint and
the intersection of blue and orange surfaces gives the curve on which to maximize or minimize.
\end{example}
\begin{figure}[H]
\centering
\includegraphics[scale=0.8]{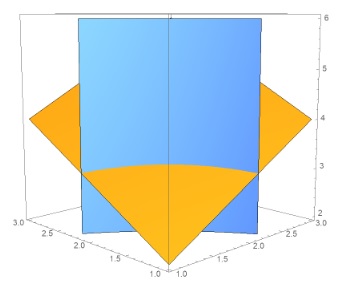}
\caption{}
\centering
\label{F1}
\end{figure}
Visually one can guess that the unique solution to $(P_{1}^{2,+})$ is located on the diagonal (more precisely, it is given by $(\overline{d}_{0},\overline{d}_{0})$ with $\overline{d}_{0}=\sqrt{7}-1$) and there are two solutions of $(P_{1}^{2,-})$ lying on the boundary (specifically, $(\overline{d}_{1},\overline{d}_{2})$ with $\overline{d}_{1}=1$, $\overline{d}_{2} =\sqrt{10}-1$ and $\overline{d}_{1} = \sqrt{10}-1$, $\overline{d}_{2}=1$).
\subsubsection{Solving $(P_{1}^{N,+})$}
In fact, what happens in  previous example can be extended to the $N-$dimensional case. More precisely, we will prove that the solution for $(P_{1}^{N,+})$ is a vector with equal coordinates:

\begin{theorem}
Let us assume  $d_{min}>0$, $\rho_{0} \geq 2$ and $N \in \left( \lambda_{0}, \rho_{0} \right] \cap \mathbb{N}.$ Then, the  unique solution to $(P_{1}^{N,+})$ has the form $\overline{d}^{N}=(\overline{d}_{0},...,\overline{d}_{0})$ with \begin{equation}
\overline{d}_{0}=\dfrac{-\alpha_{0} N + \sqrt{(\alpha_{0}N)^{2}+4\beta_{0} N \gamma_{OAR}}}{2\beta_{0} \delta N}.
\label{E5}
\end{equation}
\label{teo2}
\end{theorem}
\begin{proof}
By using the Cauchy-Schwarz inequality, we have
$$ \left( \sum_{i=1}^{N} d_{i} \right)^{2} \leq N \left( \sum_{i=1}^{N} d_{i}^{2} \right). $$
Therefore, for each feasible point it follows that
\begin{equation}
\gamma_{OAR} = \alpha_{0} \delta \sum_{i=1}^{N} d_{i} + \beta_{0} \delta^{2} \sum_{i=1}^{N} d_{i}^{2} \geq \alpha_{0} \delta \sum_{i=1}^{N} d_{i} + \dfrac{\beta_{0} \delta^{2}}{N} \left(\displaystyle \sum_{i=1}^{N} d_{i} \right)^{2}.
\label{E10}
\end{equation}
Defining $q(z)=\dfrac{\beta_{0} \delta^{2}}{N}z^{2}+\alpha_{0} \delta z - \gamma_{OAR}$, previous inequality can be rewritten as
\begin{equation}
q\left(\sum_{i=1}^{N}d_{i}\right) \leq 0.
\label{E15}
\end{equation}
Taking into account that the polynomial $q$ can be factorized in the form $q(z)=\dfrac{\beta_{0} \delta^{2}}{N}(z-z_{1})(z-z_{2})$ with $z_1 < 0 < z_2$, we know that the relation
(\ref{E15}) holds if and only if $\sum_{i=1}^{N} d_{i} \in [0,z_2]$, because all the components $d_i$ have to be positive.

Now it is clear that the maximum value is achieved when $\sum_{i=1}^{N} \overline{d}_{i} = z_2$.
Combining this fact with (\ref{E10}), we deduce that
\begin{equation}
\gamma_{OAR} = \alpha_{0} \delta \displaystyle \sum_{i=1}^{N}\overline{d}_{i} + \beta_{0} \delta^{2} \sum_{i=1}^{N}\overline{d}_{i} ^{2} \geq
\alpha_{0} \delta \sum_{i=1}^{N} \overline{d}_{i} + \dfrac{\beta_{0} \delta^{2}}{N} \left( \sum_{i=1}^{N}\overline{d}_{i}\right) ^{2} = \gamma_{OAR}.
\label{E20}
\end{equation}
Hence
\begin{equation}
\left( \sum_{i=1}^{N} \overline{d}_{i}\right) ^{2} = N \sum_{i=1}^{N} \overline{d}_{i} ^{2}.
\label{E25}
\end{equation}
In this case, Cauchy-Schwarz inequality becomes (in fact) an equality and this is true if and only if all the components are equal, i. e., $\overline{d}_{1}=...=\overline{d}_{N}.$
Therefore, $\overline{d}^{N}=(\overline{d}_{0},...,\overline{d}_{0})$ with $\overline{d}_{0}=\dfrac{z_{2}}{N}$ and (\ref{E5}) holds.

Let us emphasize that $\overline{d}_{0}$ satisfies $d_{min} \leq \overline{d}_{0} \leq d_{max}$, thanks to the hypothesis $N \in \left( \lambda_{0}, \rho_{0} \right]$. \end{proof}

\subsubsection{Solving $(P_{1}^{N,-})$}

Given $\overline{d}$ a solution of $(P_{1}^{N,-})$, since the objective function and the functions defining the restrictions are $C^{1}$, we can apply the Lagrange Multipliers Rule \cite{ref5} to deduce the existence of real numbers $\overline{\alpha} \in [0,+\infty), \overline{\lambda} \in \mathbb{R}$ and $\lbrace \overline{\mu}_{i} \rbrace_{i=1}^{2N} \subset [0,+\infty)$ verifying
\begin{equation}
\overline{\alpha}+ |\overline{\lambda}|+\sum_{i=1}^{2N}\overline{\mu}_{i} >0,
\label{l1}
\end{equation}
\begin{equation}
\overline{\alpha} \left(\begin{array}{c} 1 \\ \vdots \\ 1 \end{array} \right) + \overline{\lambda}  \left(\begin{array}{c} \alpha_{0} \delta + 2\beta_{0} \delta^{2} \overline{d}_1 \\ \vdots \\ \alpha_{0} \delta + 2\beta_{0} \delta^{2} \overline{d}_N \end{array} \right) + \left(\begin{array}{c} \overline{\mu}_{N+1}-\overline{\mu}_{1} \\ \vdots \\  \overline{\mu}_{2N}-\overline{\mu}_{N} \end{array} \right) = \left(\begin{array}{c} 0 \\ \vdots \\ 0 \end{array} \right),
\label{l2}
\end{equation}
\begin{equation}
\overline{\mu}_{i}(d_{min}-\overline{d}_i) = 0, \ \ \overline{\mu}_{i+N}(\overline{d}_i-d_{max}) = 0, \ 1 \leq i \leq N.
\label{l3}
\end{equation}

Inspired by the 2D example, we will prove that $\overline{d}$ lies on the boundary of $[d_{min},d_{max}]^N$. Let us argue by contradiction assuming that $\overline{d}_{i} \in (d_{min},d_{max})$, for all $i \in \lbrace 1,...,N \rbrace$. Then, thanks to (\ref{l3}) we deduce that $\overline{\mu}_{i}=0, \forall i \in \lbrace 1,...,2N \rbrace.$
In this case, (\ref{l2}) reads:
\begin{equation}
\overline{\alpha} \left(\begin{array}{c} 1 \\ \vdots \\ 1 \end{array} \right) + \overline{\lambda}  \left(\begin{array}{c} \alpha_{0} \delta + 2\beta_{0} \delta^{2} \overline{d}_1 \\ \vdots \\ \alpha_{0} \delta + 2\beta_{0} \delta^{2} \overline{d}_N \end{array} \right) = \left(\begin{array}{c} 0 \\ \vdots \\ 0 \end{array} \right).
\label{l33}
\end{equation}
If $\overline{\lambda} = 0$, the identity (\ref{l33}) implies that $\overline{\alpha} = 0,$ but this is not possible by (\ref{l1}). Therefore $\overline{\lambda} \neq 0$ and from (\ref{l33}) we get
$$ \overline{d}_{1}=...=\overline{d}_{N}=\dfrac{-(\overline{\alpha}+\overline{\lambda} \alpha_{0} \delta)}{2\overline{\lambda}\beta_{0} \delta^{2}}. $$
In other words, we arrive to the solution of problem $(P_{1}^{N,+})$, contradicting our initial hypothesis about $\overline{d}$.

Consequently, there exists (at least) one index $j \in \lbrace 1,...,N \rbrace$ such that $\overline{d}_{j} \in \lbrace d_{min}, d_{max} \rbrace$. Without loss of generality we can suppose that  $j=N$.
Let us see that in this case we can reduce the dimension of the optimization problem $(P_{1}^{N,-})$ by means for the following auxiliary problem:
$$ (P_{1}^{N-1,-}) \left\{
\begin{array}{ll}
\begin{aligned}
& \textrm{Minimize} \hspace{0.5cm} \sum_{i=1}^{N-1} d_{i} + \overline{d}_{N}, \\
& \textrm{subject to } d \in \mathbb{R}^{N-1} \textrm{ such that } \\
& E_{OAR}(N-1,d) = \gamma_{OAR} - \alpha_{0} \delta \overline{d}_{N} - \beta_{0} \delta^{2} \overline{d}_{N}^{2},\\
& \hspace{0.1cm} d_{min} \leq d_{i} \leq d_{max}, \hspace{0.4cm} i=1,...,N-1.
\end{aligned}
\end{array}
\right. $$
\begin{proposition}
Assume that $\overline{d}=(\overline{d}_{1},...,\overline{d}_{N})$ is a solution of $(P_{1}^{N,-})$. Then $(\overline{d}_{1},...,\overline{d}_{N-1})$ is a solution of $(P_{1}^{N-1,-})$.
\end{proposition}
\begin{proof}
Every feasible point $(d_{1},...,d_{N-1})$ for the problem $(P_{1}^{N-1,-})$ satisfies
$$ E_{OAR}(N-1,d) = \gamma_{OAR} - \alpha_{0} \delta \overline{d}_{N} - \beta_{0} \delta^{2} \overline{d}_{N}^{2}. $$
This implies that $(d_{1},...,d_{N-1},\overline{d}_{N})$ is a feasible point for $(P_{1}^{N,-}).$
Hence, using that $\overline{d}$ is a solution of $(P_{1}^{N,-})$, we get
$$ \sum_{i=1}^{N-1} \overline{d}_{i} \leq \sum_{i=1}^{N-1} d_{i}, $$
which implies that $(\overline{d}_{1},...,\overline{d}_{N-1})$ is a solution of $(P_{1}^{N-1,-})$.
\end{proof}

Arguing exactly in the same form as before with the problem $(P_{1}^{N-1,-})$, we deduce that there must be an index $j \in \lbrace 1,...,N-1 \rbrace$ such that $\overline{d}_{j} \in \lbrace d_{min},d_{max} \rbrace$ and we can reduce again the dimension of the problem, obtaining a new problem with  $N-2$ unknowns. Repeating this process several times we arrive to the final $1D$ problem:
$$ (P_{1}^{1,-}) \left\{
\begin{array}{ll}
\begin{aligned}
& \textrm{Minimize} \hspace{0.5cm} d_{1} + \sum_{i=2}^{N} \overline{d}_{i}, \\
& \textrm{subject to } d_{1} \in \mathbb{R} \textrm{ such that } \\
& \alpha_{0} \delta d_{1} + \beta_{0} \delta^{2} d_{1}^{2} = \gamma_{OAR} - \alpha_{0} \delta \sum_{i=2}^{N} \overline{d}_{i} - \beta_{0} \delta^{2} \sum_{i=2}^{N} \overline{d}_{i}^{2}, \\
& \hspace{0.1cm} d_{min} \leq d_{1} \leq d_{max}.
\end{aligned}
\end{array}
\right. $$
Clearly, it is enough to solve the quadratic equation to get the solution.

Summarizing previous results, given $N \in \left( \lambda_{0}, \rho_{0} \right] \cap \mathbb{N}$,
the solution of $(P_{1}^{N,-})$ has one of the following structures:
\begin{equation}
\overline{d}^{N}=(\underbrace{d_{min},...,d_{min}}_{K},\underbrace{d_{max},...,d_{max}}_{N-K}),
\label{E200}
\end{equation}
or
\begin{equation}
\overline{d}^{N}=(\underbrace{d_{min},...,d_{min}}_{K},d^{*},\underbrace{d_{max},...,d_{max}}_{N-K-1}),
\label{E210}
\end{equation}
with $d^* \in (d_{min},d_{max})$ being the unique positive root of the quadratic equation
\begin{equation}
\varphi_{0}(d^{*})= \gamma_{OAR} - K \varphi_{0}(d_{min})-(N-K-1) \varphi_{0}(d_{max}),
\label{E400}
\end{equation}
with $\varphi_0$ defined in (\ref{var0}).

We can characterize the unknown value $K$ as follows:
\begin{itemize}
\item [a)]  In the case (\ref{E200}), by using the equality restriction we derive that
\begin{equation}
K = \dfrac{N \varphi_{0}(d_{max})-\gamma_{OAR}}{\varphi_{0}(d_{max})-\varphi_{0}(d_{min})}.
\label{K}
\end{equation}
Of course, this holds if and only if the right hand side is a natural number or zero.

\item [b)] In the case (\ref{E210}), since $\varphi_{0}$ is an strictly increasing function in $[0,+\infty)$, we know that
$$ \varphi_{0}(d_{min}) < \varphi_{0}(d^{*}) < \varphi_{0}(d_{max}),$$
and using (\ref{E400}) we get that
$$
K \in \left(\dfrac{N\varphi_{0}(d_{max})-\gamma_{OAR}}{\varphi_{0}(d_{max})-\varphi_{0}(d_{min})}-1, \dfrac{N \varphi_{0}(d_{max})-\gamma_{OAR}}{\varphi_{0}(d_{max})-\varphi_{0}(d_{min})}\right)\cap \mathbb{N},
$$
which means that
\begin{equation}
K = \displaystyle \lfloor \dfrac{N\varphi_{0}(d_{max})-\gamma_{OAR}}{\varphi_{0}(d_{max})-\varphi_{0}(d_{min})} \rfloor.
\label{K2}
\end{equation}
\end{itemize}

Taking into account the conditions (\ref{K}) and (\ref{K2}), it is easy to conclude that the latter structure (\ref{E210}) is more frequently found in practice than  (\ref{E200}).
Previous argumentations lead us to the following result:

\begin{theorem}
Let us assume  $d_{min}>0$, $\rho_{0} \geq 2$ and $N \in \left( \lambda_{0}, \rho_{0} \right] \cap \mathbb{N}.$ Then,
a solution to problem  $(P_1^{N,-})$ is given by
\begin{itemize}
\item [a)]  $\overline{d}^N=(\underbrace{d_{min},...,d_{min}}_{K},\underbrace{d_{max},...,d_{max}}_{N-K}),$
when $K$ defined by (\ref{K}) belongs to $\mathbb{N} \cup \{0\};$ otherwise,
\item [b)] $\overline{d}^N=(\underbrace{d_{min},...,d_{min}}_{K},d^{*},\underbrace{d_{max},...,d_{max}}_{N-K-1}),$
with $K$ defined by (\ref{K2}) and $d^\star$ satisfying (\ref{E400}).
\end{itemize}
\label{teo2-}
\end{theorem}

\begin{remark}
It is not difficult to show that a solution for $(P_1^{N,-})$ is also a solution for the problem
\begin{equation}  \left\{
\begin{array}{ll}
\textrm{Minimize  } \displaystyle \sum_{i=1}^{N} d_{i,} \\
\textrm{subject to } d \in \mathbb{R}^{N} \textrm{ such that } \\
E_{OAR}(N,d) \geq \gamma_{OAR}, \\
d_{min} \leq d_{i} \leq d_{max} \textrm{, } 1 \leq i \leq N.
\end{array}
\right.
\label{PEQ-}
\end{equation}
We will use this property in the proof of Theorem \ref{TPROP3} (see Appendix $1$).
\label{R2}
\end{remark}

\subsection{Analytical solution for $(P_{1})$}
\label{SP1}

As we pointed out, a solution of $(P_{1})$ will be the pair $\left( \overline{N},\overline{d}^{\overline{N}} \right) $, where  $\overline{d}^{\overline{N}}$ denotes a solution of
$(P_1^{\overline{N}})$  from the finite set  $$\left\lbrace \left( N,\overline{d}^{N}\right): N \in \left[1, \rho_{0} \right] \cap \mathbb{N} \right\rbrace, $$ maximizing the value of $E_T(N,d)$.
In fact, combining previous results, we can avoid the calculation of most solutions for $(P_{1}^N)$ by studying its dependence with respect to $N$. This is the goal of the next results.
Let us start by studying the less frequent case: when $\lfloor \lambda_0 \rfloor = \lfloor  \rho_0 \rfloor$.
	
\begin{theorem}
Let us assume $d_{min}>0,$ $\rho_{0} \geq 2$ and $\lfloor \lambda_0 \rfloor = \lfloor  \rho_0 \rfloor$.
Then, the unique solution to problem $(P_{1})$ is given by the pair
$(\overline{N}, \overline{d}^{\overline{N}})$ with $\overline{N} = \lfloor \lambda_0 \rfloor$  and $\overline{d}^{\overline{N}} = (d_{max},...,d_{max}).$
\label{TPROP2-}
\end{theorem}
\begin{proof}
In this case the set of feasible values for $N$ is $\{1,\ldots,N_1\} \subset \mathbb{N}$ with $N_1 = \lfloor \rho_0 \rfloor = \lfloor \lambda_0 \rfloor.$
For those values of $N$, the solution for $(P_{1}^N)$ has the form $\overline{d}^{N}=(d_{max},...,d_{max})$. Among them, it is clear that in order to solve
$(P_{1})$ only the one with the largest number of components is of interest; this is attained at $N_1$. \end{proof}

We will continue to analyze the most common case: when $\lfloor \lambda_0 \rfloor < \lfloor  \rho_0 \rfloor$.
In the  the trivial case $\omega_{\delta} = 0$, the function to be minimized and the one defining the restriction are proportional.
Therefore, we can be derive the following result:

\begin{proposition}
Let us assume $d_{min}>0,$ $\rho_{0} \geq 2,$ $\lfloor \lambda_0 \rfloor < \lfloor  \rho_0 \rfloor$ and $\omega_{\delta} = 0.$
Then any feasible pair $(\overline{N},d)$  with $E_{OAR}(\overline{N},d) = \gamma_{OAR}$ is a solution to problem $(P_1)$.
In particular, the pairs $(\overline{N}, \overline{d}^{\overline{N}})$ with $\overline{N} \in \{\lceil \lambda_0 \rceil, \ldots, \lfloor  \rho_0 \rfloor\}$ and $\overline{d}^{\overline{N}} = (\overline{d}_{0},...,\overline{d}_{0}),$
where
\begin{equation}
\overline{d}_{0}=\dfrac{-\alpha_{0} \overline{N} + \sqrt{(\alpha_{0}\overline{N})^{2}+4\beta_{0} \overline{N} \gamma_{OAR}}}{2\beta_{0} \delta \overline{N}},
\label{E500}
\end{equation}
with $\overline{N}$ in the above set.
\label{PPROP2}
\end{proposition}

\begin{proof}
Due to the hypothesis $\omega_{\delta} = 0$, we deduce straightforwardly that problem $(P_1)$ is equivalent to
$$ (\tilde{P}_1) \left\{ \begin{array}{ll} \begin{aligned}
& \textrm{Maximize } E_{OAR}(N,d),  \\
& \textrm{subject to } N \in \mathbb{N}, d \in \mathbb{R}^{N} \textrm{ such that  } \\
& E_{OAR}(N,d) \leq \gamma_{OAR}, \\
& d_{min}\leq d_{i}\leq d_{max}, \ i=1,...,N,
\end{aligned}
\end{array}
\right.$$
Obviously, the maximum value is reached when the restriction becomes an equality.
This can be achieved in several ways, such as the treatments with equal doses described in the proposition statement.
Let us emphasize that $\overline{d}_{0} \in [d_{min},d_{max}]$ if and only if $\overline{N} \in [\lambda_0,\rho_0].$  \end{proof}

\begin{theorem}
Let us assume $d_{min}>0,$ $\rho_{0} \geq 2,$ $\lfloor \lambda_0 \rfloor < \lfloor  \rho_0 \rfloor$ and $\omega_{\delta} > 0.$
Then, the unique solution to problem $(P_{1})$ is given by the pair
$\left( \overline{N},\overline{d}^{\overline{N}} \right) $ with $\overline{N} = \lfloor \rho_0 \rfloor$ and
$\overline{d}^{\overline{N}}=(\overline{d}_{0},...,\overline{d}_{0})$, with $\overline{d}_{0}$ given by (\ref{E500}).
\label{TPROP2}
\end{theorem}
\begin{proof}
Here, the set of feasible values for $N$ is $\{1,\ldots,N_2\} \subset \mathbb{N}$ with $N_2 = \lfloor \rho_0 \rfloor.$
Arguing as in previous theorem, among the small values (i. e. $N  \in  \{1,\ldots,N_1\},$ with $N_1 = \lfloor \lambda_0 \rfloor,$
for solving $(P_{1})$ we only retain $N=N_1$ and $\overline{d}^{N_1}=(d_{max},...,d_{max})$.
For the other values, i. e. $N  \in  \{N_1+1,\ldots,N_2\},$ since $\omega_{\delta} > 0$, the corresponding solution for $(P_{1}^N)$  is given by $\overline{d}^{N}=(\overline{d}_{0},...,\overline{d}_{0})$ with $\overline{d}_{0}$ defined in (\ref{E5}).
In order to study the dependence with respect to $N$ for these values, thanks to Proposition \ref{PROP1}, it is enough to consider the auxiliary function
$$\psi(N)= N \overline{d}_{0}=\dfrac{-\alpha_{0} N + \sqrt{(\alpha_{0} N)^{2}+4\beta_{0} N \gamma_{OAR}}}{2\beta_{0} \delta}.$$
Here, it follows easily that $\psi(N)$ is an strictly increasing function and then, it will take its maximum value in $\{N_1+1,\ldots,N_2\}$ at $N_2.$

Finally, we will derive that $(N_2,\overline{d}^{N_2})$ is the unique solution to problem $(P_{1})$ by showing that
\begin{equation}
E_T(N_1,\overline{d}^{N_1}) < E_T(N_2,\overline{d}^{N_2}).
\label{E530}
\end{equation}

To that end, let us consider the linear function
\[ H(x) = N_2(x\overline{d}_0+\overline{d}_0^2)-N_1(x d_{max}+d_{max}^2), \ \ \ x \in [\dfrac{\alpha_0}{\beta_0 \delta}, +\infty),\]
with
\begin{equation}
\overline{d}_0= \dfrac{-\alpha_{0} N_2 + \sqrt{(\alpha_{0} N_2)^{2}+4\beta_{0} N_2 \gamma_{OAR}}}{2\beta_{0} \delta N_2}.
\label{E550}
\end{equation}

Using that $N_1 \varphi_0(d_{max}) \leq \gamma_{OAR} = N_2\varphi_0(\overline{d}_0)$ by the admissibility,
we get that $H\left(\dfrac{\alpha_0}{\beta_0 \delta}\right) \geq 0$. Also, taking into account (\ref{E550}),
it can also be checked that $H'(x) = N_2\overline{d}_0-N_1 d_{max} > 0,$ because $N_1 < N_2$.
Then, from the assumption $\omega_{\delta} >0$ (see (\ref{Ew})), it follows
that $H\left(\dfrac{\alpha_T}{\beta_T}\right) > 0$, which is equivalent to (\ref{E530}).
\end{proof}

In the case $\omega_{\delta} < 0,$ the situation is more complicated and it is detailed in the next result. The proof is a little bit technical and is postponed to the Appendix $1$:

\begin{theorem}
Let us assume $d_{min}>0,$ $\rho_{0} \geq 2,$ $\lfloor \lambda_0 \rfloor < \lfloor  \rho_0 \rfloor$ and $\omega_{\delta} < 0$.  Then, a solution to problem $(P_{1})$ is given by one of the following pairs:
\begin{itemize}
\item [i)] $\left( \overline{N},\overline{d}^{\overline{N}} \right) $ with $\overline{N} = \lfloor \lambda_0 \rfloor$ and
$\overline{d}^{\overline{N}}=(d_{max},...,d_{max}),$
\item [ii)] $\left( \overline{N},\overline{d}^{\overline{N}} \right) $ with $\overline{N} = \lceil \lambda_0 \rceil$ and
$\overline{d}^{\overline{N}}=(\underbrace{d_{min},...,d_{min}}_{K},\underbrace{d_{max},...,d_{max}}_{\overline{N}-K}),$
when $K$ defined by (\ref{K}) with $N = \overline{N}$ belongs to $\mathbb{N} \cup \{0\},$ or
\item [iii)] $\left( \overline{N},\overline{d}^{\overline{N}} \right) $ with $\overline{N} = \lceil \lambda_0 \rceil$ and
$\overline{d}^{\overline{N}}=(\underbrace{d_{min},...,d_{min}}_{K},d^{*},\underbrace{d_{max},...,d_{max}}_{\overline{N}-K-1}),$
where $K$ is defined in (\ref{K2}) and $d^\star$ satisfies (\ref{E400}) with $N = \overline{N}.$
\end{itemize}
\label{TPROP3}
\end{theorem}

\begin{remark}
\begin{itemize}
\item [a)] Let us strongly highlight that all expressions included in Theorems \ref{TPROP2-}-\ref{TPROP3} and Proposition \ref{PPROP2} can be explicitly calculated from the initial data of the problem.
Moreover, the calculations to implement these solutions are elementary and can be carried out using a pocket calculator.
\item [b)] On the other hand, when $\omega_{\delta} >0$, the optimal value of $N$ is the largest one within its range of possibilities (i.e. it is a hyperfractionated type treatment) with equal doses,
while in the case $\omega_{\delta} < 0$ the optimal value is the smallest one (i.e. it is a hypofractionated type treatment). In this last case, let us stress that not all doses
have to be equal or large; in fact, some of them may be minimum. As far as we know, this structure is not usually cited in the specialized literature.
\item [c)] One interesting case appears when $\dfrac{\alpha_T}{\beta_T} < \dfrac{\alpha_0}{\beta_0},$ because then $\omega_{\delta} < 0$ for all $\delta \in (0,1]$  and the optimal regimen is always of hypofractionated type, independent of the technology used and the geometry of the tumor. In practice this condition holds in some special cases, such as the prostate tumor, where $\dfrac{\alpha_T}{\beta_T} \approx 1.5 \ Gy,$ while $\dfrac{\alpha_0}{\beta_0} = 2 \ Gy$, see \cite{Mizuta2012} and \cite{ref4}.
\item [d)] After Remark \ref{R1a}-b) (see also the Subsection \ref{EQTR} below), it is clear that the hypofractionated case (associated with $\omega_{\delta} < 0$) is very convenient in the practice.
Assuming that the other parameters are set, the condition $\omega_{\delta} < 0$ can always be achieved by taking $\delta$ close enough to $0$. This last fact is related to increasing the precision of the radiotherapy process (for instance, by using cutting-edge technology).
\item [e)] A related problem to $(P_1)$ is studied in \cite{Bortfeld2015} and \cite{Bruni2019}, where the number of dose fractions $N$ is also an unknown, jointly with $d$.
The framework for that problems is more general, because a repopulation term is included in the objective function, but only the lower bound $d_{min} = 0$ is assumed.
Furthermore, the determination of the optimal value for $N$ is carried out in \cite{Bruni2019} by means of numerical simulations, while in \cite[Theorem 2]{Bortfeld2015} it is done explicitly
and the value $\overline{N} = 1$ is obtained when $\omega_{\delta} < 0$. In this last case, the single dose could be too large in practice (remember that no upper bound is imposed in \cite{Bortfeld2015}) and then more fractions would have to be tried until an acceptable one is found.
\end{itemize}
\label{R3}
\end{remark}
Next, we illustrate the general process with a particular example:
\begin{example}
Let us consider the following parameters taken from a typical clinical situation: $\alpha_{T} = 0.05 \ Gy^{-1}, \beta_{T}=0.005 \ Gy^{-2}, \alpha_{0} = 0.04 \ Gy^{-1}, \beta_{0}=0.02 \ Gy^{-2},$ see \cite{Mizuta2012}, together with $d_{min}=1 \ Gy,$ and $d_{max}=6 \ Gy.$ Then, the problem reads
$$ (P_{11}) \left\{
\begin{array}{ll}
\textrm{Maximize  }  E_{T}(N,d) = 0.05 \displaystyle \sum_{i=1}^{N} d_{i} + 0.005 \sum_{i=1}^{N} d_{i}^{2}, \\
\textrm{subject to } N \in \mathbb{N}, d_{i} \in \mathbb{R}, \\
 \displaystyle E_{OAR}(N,d) = 0.04 \delta \sum_{i=1}^{N} d_{i} + 0.02 \delta^{2}\sum_{i=1}^{N} d_{i}^{2} \leq \gamma_{OAR}, \\
1 \leq d_{i} \leq 6, i=1,...,N.
\end{array} \right.  $$

\begin{itemize}
\item [i)] For $\delta=0.3$ and $\gamma_{OAR}=0.78$, we have $\lambda_0 \approx 5.7$, $\rho_{0} \approx 56.52$ and $\omega_{\delta} \approx  3.33 > 0$. Among the  values, $N \in \{6,7,8,\dots,54,55,56\}$, we have proved (see Theorem \ref{TPROP2}) that the biggest one, $N=56$, and the solution of $(P_{11}^{56})$ (that here is the hyperfractionated $\overline{d}^{56}=(1.008,...,1.008)$) provides the solution of $(P_{11})$; in fact, $E_{T}(56,\overline{d}^{56}) \approx 3.107$. We can easily check that with the standard protocol $\tilde{d}_S^{25}=(2,...,2)$ we get $E_{T}(25,\tilde{d}_S^{25})=3$, and therefore there is about $3.5\%$ gain in terms of effect on the tumor, while the efficiency regarding OAR is the same ($E_{OAR}(56,\overline{d}^{56})=E_{OAR}(25,\tilde{d}_S^{25})=0.78$).

 On the other hand, the hypofractionated radiotherapy given by $\tilde{d}_2^{15} = (2.67,...,2.67)$ produces $E_{T}(15,\tilde{d}_2^{15}) \approx 2.54,$ although the damage on OAR is also lower: $E_{OAR}(15,\tilde{d}_2^{15}) \approx 0.67$. These last treatments are mentioned in \cite[pg. 16]{RDF2019} in connection with breast cancer.

    Of course, here we are only taking into account the mathematical point of view. In clinical practice, other factors such as patient inconvenience and additional cost may advise the use of fewer doses, if the difference in terms of efficiency is considered small.

\item [ii)] For $\delta=0.1$ and $\gamma_{OAR}=0.22$, we calculate $\lambda_0 \approx 7.05$, $\rho_{0} \approx 52.38$ and $\omega_{\delta}=-10 < 0$. In this case, the solution for $(P_{11})$ is given by $(\overline{N},\overline{d})$ with $\overline{N}=8, \overline{d}=(1,d^{*},\underbrace{6,...,6}_{6})$ and $d^{*} \approx 5.588 \ Gy$ (see Theorem \ref{TPROP3}-iii)), having $E_{T}(8,\overline{d}) \approx 3.37$. Recall that this is a hypofractionated type treatment. Just for comparison reasons, let us mention that the solution of $(P_{11}^{7})$ is $\tilde{d}_1^7 = (6,...,6)$  and the solution of $(P_{11}^{9})$ is $\tilde{d}_2 = (1,1, d_2^{*},\underbrace{6,...,6}_{6})$ with $d_2^{*} \approx 4.9 \ Gy$ producing $E_{T}(7,\tilde{d}_1^7) \approx 3.36$ and $E_{T}(9,\tilde{d}_2) \approx 3.355$, that are smaller than $E_{T}(8,\overline{d}) $ as expected.

\item [iii)] Let us emphasize that the difference between ``few" and ``many" doses is relative to each particular problem and not an absolute classification. For instance, in the problem $(P_{11})$ with $\delta=0.3$ and $\gamma_{OAR}=0.1$, the solution is given by $(\overline{N},\overline{d})$ with $\overline{N}=7, \overline{d}^7=(1.031,\ldots,1.031)$ which corresponds to the hyperfractionated case (because $N \in \{1,\ldots,7\}$), although the number of delivered doses is lower than in the previous hypofractionated treatment, see $ii)$.
\end{itemize}
\label{ejemplo1}
\end{example}

For $\omega_{\delta} < 0$, in most practical situations the solution is the one presented in Theorem \ref{TPROP3}-$iii)$, but the alternatives $i)$ and $ii)$ can also appear
as we show in the following example:

\begin{example}
Let us take the following parameters: $\alpha_{T} = 0.08 \ Gy^{-1},$ $\beta_{T}=0.02 \ Gy^{-2},$ $\alpha_{0}=0.01 \ Gy^{-1}, \beta_0 = 0.001 \ Gy^{-2},  \delta = 1,$ $d_{min} = 1 \ Gy,$ and $d_{max}=6 \ Gy.$ Here, $\omega_{\delta} = -6 < 0$ and the problem under consideration is
$$ (P_{12}) \left\{
\begin{array}{ll}
\textrm{Maximize  } E_{T}(N,d) = 0.08\displaystyle \sum_{i=1}^{N} d_{i} + 0.02\sum_{i=1}^{N} d_{i}^{2}, \\
\textrm{subject to } N \in \mathbb{N}, d_{i} \in \mathbb{R}, \\
0.01\displaystyle \sum_{i=1}^{N} d_{i} + 0.001 \sum_{i=1}^{N} d_{i}^{2} \leq \gamma_{OAR}, \\
1 \leq d_{i} \leq 6, i=1,...,N.
\end{array}
\right.  $$
\end{example}

For $\gamma_{OAR}=0.961,$ we get that $\lambda_0 \approx 10.01$ and $\rho_0 \approx 87.36$. For $N = \lfloor \lambda_0 \rfloor = 10$, the solution of $(P_{12}^{10})$ is given by $\overline{d}^{10}=(6,...,6)$.
While for $N = \lceil \lambda_0 \rceil = 11$, the solution of $(P_{12}^{11})$ is given by $\overline{d}^{11}=(1,d^*,\underbrace{6,...,6}_{9})$, with $d^* \approx 5.53565 \ Gy$.
Since $E_T(10,\overline{d}^{10}) =  12 >  11.956 \approx  E_T(11,\overline{d}^{11})$, the solution of $(P_{12})$ is given by $\overline{N} =10$ and $\overline{d}^{10}$.
This shows that sometimes the option $i)$ is the valid one.

Finally, taking $\gamma_{OAR}=0.971,$ we have $\lambda_0 \approx 10.11,$ $N = \lceil \lambda_0 \rceil = 11$ and
$K = \dfrac{N \varphi_{0}(d_{max})-\gamma_{OAR}}{\varphi_{0}(d_{max})-\varphi_{0}(d_{min})} = 1.$
Then, the solution of $(P_{12}^{11})$ is given by $\overline{d}^{11}=(1,\underbrace{6,...,6}_{10}).$
Now, the solution of $(P_{12})$ is also given by $\overline{N} = 11$ and $\overline{d}^{11},$ because $E_T(11,\overline{d}^{11}) =  12.1$.
So, in this case Theorem \ref{TPROP3}-$ii)$ holds.

Once we have obtained the analytical expressions for the solution of problem $(P_1)$, we can deduce very easily its dependence with respect to the parameters defining the problem.
This will help us to know how to adjust these parameters in order to achieve a desired solution. Let us show a result in the direction:

\begin{corollary}
Let us assume $d_{min}>0,$ $\rho_{0} \geq 2$ and $(\overline{N}, \overline{d}^{\overline{N}})$ is a solution of $(P_1)$. Then,
$\overline{N}$ is an increasing function of $\gamma_{OAR}$, decreasing with respect to $\alpha_0, \beta_0$ and $\delta$ and independent of $\alpha_T$ and $\beta_T$. Moreover,
\begin{itemize}
\item [a)] When $\omega_{\delta} > 0,$ $\overline{N}$ is also decreasing with respect to $d_{min}$ and independent of $d_{max}$.
\item [b)] When $\omega_{\delta} < 0,$ $\overline{N}$ is also decreasing with respect to $d_{max}$ and independent of $d_{min}$.
\end{itemize}
\label{CO16}
\end{corollary}

\begin{proof}
It is just a consequence of the expressions $\overline{N} = \lfloor \rho_0 \rfloor$ with $\rho_{0}$ given by (\ref{rho0}), when $\omega_{\delta} > 0,$
and $\overline{N} = \lfloor \lambda_0 \rfloor$ or $\lceil \lambda_0 \rceil$ with $\lambda_0$ given by (\ref{lam0}), when $\omega_{\delta} < 0.$
\end{proof}

In Table \ref{solP1}, we summarize the resolution of problem $(P_1)$ in algorithm form, for the reader's convenience.

\begin{table}
\begin{center}
\begin{tabular}{| c | }
\hline
ALGORITHM FOR SOLVING $(P_1)$ \\
\hline
\hline
DATA: $\alpha_{T}, \beta_{T}, \alpha_{0}, \beta_{0}, d_{min},
d_{max}, \delta$ and $\gamma_{OAR}$ \\
(all positive, $d_{min} < d_{max}$ and $\delta \leq 1$) \\
\hline
CALCULATE:  $\omega_{\delta} = \dfrac{\alpha_T}{\beta_T} - \dfrac{\alpha_0}{\beta_0 \delta},$
$\lambda_{0} = max\left\lbrace  1, \dfrac{\gamma_{OAR}}{\varphi_{0}(d_{max})} \right\rbrace$ and \\
$\rho_{0} = \dfrac{\gamma_{OAR}}{\varphi_{0}(d_{min})},$
with $\varphi_{0}(r)=\alpha_{0} \delta r + \beta_{0} \delta^{2} r^{2}.$ \\
\hline
IF $\rho_0 < 1,$ $(P_1)$ has NO SOLUTION. \\
\hline
IF $\rho_0 = 1$, the pair $(\overline{N}, \overline{d}^{\overline{N}})  = (1,d_{min})$ is the UNIQUE SOLUTION of $(P_{1})$.\\
\hline
IF $\rho_0 \in (1,2)$, the UNIQUE SOLUTION of $(P_1)$ is the pair \\ $(\overline{N}, \overline{d}^{\overline{N}})  = (1,\min{\{d_{max},\overline{d}_{0}\}}),$
with $\overline{d}_{0}=\dfrac{-\alpha_{0} + \sqrt{\alpha^2_{0}+4\beta_{0} \gamma_{OAR}}}{2\beta_{0} \delta}$.\\
\hline
IF $\rho_0 \geq 2$ and $\lfloor \lambda_0 \rfloor = \lfloor  \rho_0 \rfloor$, the UNIQUE SOLUTION of $(P_1)$ \\ is the pair
$(\overline{N}, \overline{d}^{\overline{N}})$ with $\overline{N} = \lfloor \lambda_0 \rfloor$  and $\overline{d}^{\overline{N}} = (d_{max},...,d_{max}).$ \\
\hline
IF $\rho_0 \geq 2,$ $\lfloor \lambda_0 \rfloor < \lfloor  \rho_0 \rfloor$ and $\omega_{\delta} > 0,$ the UNIQUE SOLUTION of $(P_1)$ \\ is the pair
$(\overline{N}, \overline{d}^{\overline{N}})$  with  $\overline{N} =\lfloor \rho_0 \rfloor$ and $\overline{d}^{\overline{N}} = (\overline{d}_{0},...,\overline{d}_{0}),$
\\  where $\overline{d}_{0}=\dfrac{-\alpha_{0} \overline{N} + \sqrt{(\alpha_{0}\overline{N})^{2}+4\beta_{0} \overline{N} \gamma_{OAR}}}{2\beta_{0} \delta \overline{N}}$.
\\ \hline
IF $\rho_0 \geq 2,$ $\lfloor \lambda_0 \rfloor < \lfloor  \rho_0 \rfloor$ and $\omega_{\delta} < 0,$ take $\overline{N}_1 = \lceil \lambda_0 \rceil$ and \\
CALCULATE: $M = \dfrac{\overline{N}_1 \varphi_{0}(d_{max})-\gamma_{OAR}}{\varphi_{0}(d_{max})-\varphi_{0}(d_{min})}$. \\
IF $M  \in \mathbb{N} \cup \{0\}$, take $K = M$ and $\overline{d}_1^{\overline{N}_1} = (\underbrace{d_{min},...,d_{min}}_{K},\underbrace{d_{max},...,d_{max}}_{\overline{N}_1-K})$. \\
IF $M  \not\in \mathbb{N} \cup \{0\}$, take $K = \lfloor M \rfloor$ and $\overline{d}_1^{\overline{N}_1} = (\underbrace{d_{min},...,d_{min}}_{K},d^{*},\underbrace{d_{max},...,d_{max}}_{\overline{N}_1-K-1}),$ \\ with $d^{*} > 0$ and $\varphi_{0}(d^{*})= \gamma_{OAR} - K \varphi_{0}(d_{min})-(\overline{N}_1-K-1) \varphi_{0}(d_{max})$. \\
Also take $\overline{N}_2 = \lfloor \lambda_0 \rfloor$  and $\overline{d}_2^{\overline{N}_2} = (d_{max},...,d_{max}).$ \\
CALCULATE: $E_T( \overline{N}_1,\overline{d}_1^{\overline{N}_1})$  and $E_T( \overline{N}_2,\overline{d}_2^{\overline{N}_2}).$ \\
A SOLUTION of $(P_1)$ is the pair $(\overline{N}, \overline{d}^{\overline{N}})$ that maximizes $E_T$  between them.\\
\hline
IF $\rho_0 \geq 2,$ $\lfloor \lambda_0 \rfloor < \lfloor  \rho_0 \rfloor$ and $\omega_{\delta} = 0,$ ANY FEASIBLE PAIR $(\overline{N},d)$  such that \\
$E_{OAR}(\overline{N},d) = \gamma_{OAR}$ is a SOLUTION for $(P_1)$. \\
In particular, the pairs $(\overline{N}, \overline{d}^{\overline{N}})$ with $\overline{N} \in \{\lceil \lambda_0 \rceil, \ldots, \lfloor  \rho_0 \rfloor\}$ and $\overline{d}^{\overline{N}} = (\overline{d}_{0},...,\overline{d}_{0}),$
\\  where $\overline{d}_{0}=\dfrac{-\alpha_{0} \overline{N} + \sqrt{(\alpha_{0}\overline{N})^{2}+4\beta_{0} \overline{N} \gamma_{OAR}}}{2\beta_{0} \delta \overline{N}}.$\\
\hline
\end{tabular}
\caption{Complete solution for problem $(P_1)$ in algorithmic form.}\label{solP1}
\end{center}
\end{table}

\section{Minimizing the effect of radiation on the organs at risk}
In this section, we will consider a problem closely related to that of the previous section: the goal of this second issue will be to determine the best strategy to minimize the effect of radiation on the organs at risk, while maintaining a minimum effect of radiation on the tumor. It is clear that this approach can be interesting (at least) for palliative therapies. Mathematically we formulate it in the following way:
$$ (P_2)\left\{ \begin{array}{ll}
\begin{aligned} & \textrm{Minimize } E_{OAR}(N,d),  \\
& \textrm{subject to } N \in \mathbb{N}, \\
& d \in \mathbb{R}^{N} \textrm{ such that  } \\
& E_{T}(N,d) \geq \gamma_{T}, \\
& d_{min}\leq d_{i}\leq d_{max} \textrm{,  } i=1,...,N,
\end{aligned}
\end{array}
\right. $$
where $E_{OAR}(N,d)$ is given by (\ref{EOAR}), $E_{T}(N,d)$ is defined in (\ref{ET}) and $\gamma_{T}$ is a given positive parameter.
Of course, this is also a mixed optimization problem with $N+1$ unknowns: the number of radiation doses, $N \in \mathbb{N}$, and the value of the $N$ doses, $d_{i} \in \mathbb{R}, 1 \leq i \leq N$.

This problem has recently been studied in the outstanding work \cite{Mizuta2012}, but with fixed $N$ and only imposing the non-negativity constraint for the doses. Moreover, in \cite{Mizuta2012} it is also remarked that ``The real interest of the present approach would be the determination of the optimum solution for $N$ in clinical practice". As an intermediate step, we have achieved here the expression of the optimal value for $N$ in terms of the parameters of the problem in this particular setting, see Table \ref{solP2} for a detailed description, depending on the case.

As we will see, the study for problem  $(P_2)$ can be carried out following the same argumentation to that of $(P_1)$ with minor differences.

\subsection{Existence of solution for $(P_2)$}

In the sequel we will denote
\begin{equation}
\varphi_{T}(r)=\alpha_T r + \beta_T r^{2},
\label{varT}
\end{equation}
\begin{equation}
\lambda_{T} = max\left\lbrace  1, \dfrac{\gamma_T}{\varphi_{T}(d_{max})}\right\rbrace,
\label{lamT}
\end{equation}
and
\begin{equation}
\rho_{T} = max\left\lbrace  1, \dfrac{\gamma_T}{\varphi_{T}(d_{min})} \right\rbrace.
\label{rhoT}
\end{equation}

Our first observation concerns the existence of solution for $(P_2)$:

\begin{theorem}	
Let us assume $d_{min}>0$. Then, the problem $(P_{2})$ has (at least) one solution.
\end{theorem}
\label{teo20}
\begin{proof}
It is analogous to that of Theorem \ref{teo1}, although here there are infinite feasible values for $N$:
combining the restrictions, those such that $N \in [\lambda_{T},+\infty) \cap \mathbb{N}$.
We will begin by showing that for each fixed feasible value $N$, the associated problem $(P_{2}^{N})$ has a solution, where
$$(P_{2}^{N}) \left\{
\begin{array}{ll}
\begin{aligned}
& \textrm{Minimize} \hspace{0.4cm} \tilde{E}^N_{OAR}(d) = \alpha_{0} \delta \sum_{i=1}^{N} d_{i} +\beta_{0} \delta^{2} \sum_{i=1}^{N}d_{i}^{2},  \\
& \textrm{subject to } d \in \mathbb{R}^{N} \textrm{ such that }\\
& E_{T}(N,d) \geq \gamma_{T}, \\
& d_{min}\leq d_{i}\leq d_{max} \textrm{,  } i=1,...,N.
\end{aligned}
\end{array}
\right. $$
For large values of $N$, specifically for $N \geq \rho_{T}$, the solution of $(P_{2}^{N})$ is the trivial one with minimum values $d_{min}$. Among them only the smallest value of $N$ could have practical interest, i.e., $\lceil \rho_{T} \rceil$. For the other values, when they exist, that is for $N \in \left[ \lambda_{T}, \rho_{T} \right) \cap \mathbb{N}$, the existence of solution for $(P_{2}^{N})$ is a consequence of Weierstrass Theorem, once more because the objective function is continuous and the feasible set is compact. Therefore, for each value of $N$ in that interval, let us consider a global solution for the problem $(P_{2}^{N})$ that we will denote $\overline{d}^{N}$. Again, it is enough to take the pair $\left( \overline{N},\overline{d}^{\overline{N}} \right) $ from the finite set $$\left\lbrace \left( N,\overline{d}^{N}\right): N \in \left[ \lambda_{T}, \lceil \rho_{T} \rceil \right] \cap \mathbb{N} \right\rbrace, $$ that minimizes the value of $E_{OAR}(N,d)$  as a solution to the problem $(P_{2})$.
\end{proof}

\begin{remark}
\begin{itemize}
\item [a)] Once more, except if all the coordinates of $\overline{d}^{\overline{N}}$ are equal, the solution will not be unique, because two different coordinates can be permuted to generate a different solution.
\item [b)] When $\omega_{\delta} > 0$ (see (\ref{Ew})), the hypothesis $d_{min} > 0$ is necessary for proving the existence of solution for $(P_2)$, as we can see through the following example:
\begin{example}
$$ (P_{20}) \left\{
\begin{array}{ll}
\textrm{Minimize  }  \displaystyle E_{OAR}(N,d)= \sum_{i=1}^{N} d_{i} + \sum_{i=1}^{N} d_{i}^{2}, \\
\textrm{subject to } N \in \mathbb{N}, d_{i} \in \mathbb{R}, \\
\displaystyle E_{T}(N,d) = 2\sum_{i=1}^{N} d_{i} + \sum_{i=1}^{N} d_{i}^{2} \geq 10, \\
0 \leq d_{i} \leq 1, i=1,...,N.
\end{array} \right.  $$
 \end{example}
Let us consider the sequence given by
\[ d^N = (d_{0N},...,d_{0N}), \ \ \mbox{ with } \ \ d_{0N} = -1 + \sqrt{1+\dfrac{10}{N}}. \ \]
It is easy to check that it is feasible for $N \geq 4$,
\[ E_{T}(N,d^N) =10, \ \  E_{OAR}(N,d^N) = 10 - Nd_{0N} \longrightarrow 5, \ \ \mbox{ as } N \rightarrow +\infty.\]
Now, we can deduce that $ (P_{20})$ can not have solution $(\overline{N},\overline{d})$ because if $E_{OAR}(\overline{N},\overline{d})  =  5$
together with the restriction $E_{T}(\overline{N},\overline{d}) \geq 10$ we get $\sum_{i=1}^{\overline{N}} \overline{d}_{i} \geq 5,$
but this is incompatible with $E_{OAR}(\overline{N},\overline{d})  =  5.$
\item [c)] In contrast, when $\omega_{\delta} < 0$, it is easy to show that problem $(P_2),$ with only the lower bound constraints $d_i \geq 0$, has as solution $(\overline{N},\overline{d})$ with $\overline{N} = 1$ and $\overline{d} = \dfrac{-\alpha_{T} + \sqrt{(\alpha_{T})^{2}+4\beta_{T} \gamma_T}}{2\beta_T },$ see \cite{Mizuta2012}.

\item [d)] There are some particular cases in which the solution of $(P_{2})$ can be determined from previous argumentations very easily. For instance, when $\rho_T = 1$, because then $\left( \overline{N},\overline{d}^{\overline{N}} \right)  = (1,d_{min})$ is the only admissible pair. Also when $\rho_T > 1$ and $\lceil \lambda_T \rceil = \lfloor  \rho_T \rfloor$, because only the large values for $N$ are feasible (i.e., those verifying $N \geq \rho_{T}$) and consequently $(\overline{N}, \overline{d}^{\overline{N}})$ with $\overline{N} = \lceil \rho_T \rceil$  and $\overline{d}^{\overline{N}} = (d_{min},...,d_{min})$ is the solution of $(P_{2})$.
\end{itemize}
\label{R100}
\end{remark}

When $ N \in [ \lambda_{T}, \rho_{T})  \cap \mathbb{N}$, we know that $d^N = (d_{min},...,d_{min})$ is not a solution of $(P_{2}^{N})$, because it is not even feasible. Hence, we can simplify the problem $(P_{2}^{N})$ arguing in a similar way as in the proof of Theorem \ref{teoAct}.

\begin{theorem}
Let us assume  $d_{min}>0$ and $N \in \left[ \lambda_{T}, \rho_{T} \right) \cap \mathbb{N}$. Then, the inequality constraint  of the problem $(P_{2}^{N})$ has to be  active at $\overline{d}^{N}$, being $\overline{d}^{N}$ a solution of $(P_{2}^{N})$.
\label{desig}
\end{theorem}

From now on, the restriction will be taken as one of equality, that is,
$$ \alpha_{T} \sum_{i=1}^{N} d_{i} + \beta_{T}\sum_{i=1}^{N} d_{i}^{2} = \gamma_{T}. $$

Again, applying the same procedure as for $(P_1)$ in Section \ref{Smax} we have

$$ \sum_{i=1}^{N} d_{i}^{2}=\dfrac{1}{\beta_{T}} \left[ \gamma_{T} - \alpha_{T} \sum_{i=1}^{N} d_{i} \right],  $$
and the objective function will read

\begin{equation}
\tilde{E}^N_{OAR}(d)=\left[  \alpha_{0} - \dfrac{\beta_{0}\alpha_{T} \delta }{\beta_{T}} \right] \delta \sum_{i=1}^{N} d_{i} + \dfrac{\beta_{0}\delta^2 \gamma_{T}}{\beta_{T}}.
\label{e1}
\end{equation}

Now, it is clear that we can simplify the formulation of the problem $(P_{2}^{N})$, as follows:

\begin{proposition}
Let us assume  $d_{min}>0$ and $N \in \left[ \lambda_{T}, \rho_{T} \right) \cap \mathbb{N}$.
\begin{itemize}
\item [i)] If  $\omega_{\delta}>0$, then $(P_{2}^{N})$ is equivalent to
$$ (P_{2}^{N,+}) \left\{
\begin{array}{ll}
\textrm{Maximize  } \displaystyle \sum_{i=1}^{N} d_{i}, \\
\textrm{subject to } d \in \mathbb{K}_{2}^{N},
\end{array}
\right. $$
where
\begin{equation}
\mathbb{K}_{2}^{N} = \lbrace d \in \mathbb{R}^{N} : E_{T}(N,d) = \gamma_{T}, d_{min} \leq d_{i} \leq d_{max}, 1 \leq i \leq N \rbrace.
\label{E695}
\end{equation}

\item [ii)] If $\omega_{\delta}<0$, then $(P_{2}^{N})$ is equivalent to
$$ (P_{2}^{N,-}) \left\{
\begin{array}{ll}
\textrm{Minimize  } \displaystyle \sum_{i=1}^{N} d_{i}, \\
\textrm{subject to } d \in \mathbb{K}_{2}^{N}.
\end{array}
\right. $$
\item [iii)] If $\omega_{\delta}=0,$ then every feasible point for $(P_{2}^{N})$ is a solution.
\end{itemize}
\label{PROP2}
\end{proposition}

\begin{proof}
It is enough to take into account that $\alpha_{0} - \dfrac{\beta_{0}\alpha_{T}\delta}{\beta_{T}} = - \beta_{0} \delta \omega_{\delta},$ where $\omega_{\delta}$ is defined in (\ref{Ew}).
\end{proof}

Once we have seen that $(P_{1}^{N,+})$ and $(P_{2}^{N,+})$ are essentially the same problem (resp. $(P_{1}^{N,-})$ and $(P_{2}^{N,-})$), we can ``translate"  the results obtained in section
\ref{SP1N} to the current context as follows:

\begin{theorem}
Let us assume $d_{min}>0$ and $N \in \left[ \lambda_{T}, \rho_{T}  \right) \cap \mathbb{N}$. Then,
\begin{itemize}
 \item [i)] the unique solution to $(P_{2}^{N,+})$ is given by $\overline{d}^{N}=(\overline{d}_{1},...,\overline{d}_{1})$ with
\begin{equation}
\overline{d}_{1}=\dfrac{-\alpha_{T} N + \sqrt{(\alpha_{T}N)^{2}+4\beta_{T} N \gamma_{T}}}{2\beta_{T} N}.
\label{E700}
\end{equation}

\item [ii)] a solution for $(P_{2}^{N,-})$ has one of the following forms:
\begin{equation}
\overline{d}^{N}=(\underbrace{d_{min},...,d_{min}}_{K},\underbrace{d_{max},...,d_{max}}_{N-K}),
\label{E710}
\end{equation}
with
\begin{equation}
K = \dfrac{N\varphi_{T}(d_{max}) - \gamma_{T}}{\varphi_{T}(d_{max})-\varphi_{T}(d_{min})} \in  \mathbb{N}  \cup \{0\}, \ \ \mbox{ or }
\label{E720}
\end{equation}
\begin{equation}
\overline{d}^{N}=(\underbrace{d_{min},...,d_{min}}_{K},d^{*},\underbrace{d_{max},...,d_{max}}_{N-K-1}),
\label{E730}
\end{equation}
with
\begin{equation}
K = \lfloor \dfrac{N\varphi_{T}(d_{max}) - \gamma_{T}}{\varphi_{T}(d_{max})-\varphi_{T}(d_{min})}\rfloor,
\label{E740}
\end{equation} and $d^* \in (d_{min}, d_{max})$ satisfying
\begin{equation}
\varphi_{T}(d^{*}) = \gamma_{T}  - K \varphi_{T}(d_{min}) - (N-K-1) \varphi_{T}(d_{max}).
\label{E750}
\end{equation}

\end{itemize}
\label{teo22}
\end{theorem}

\subsection{Analytical solution for $(P_{2})$}
\label{SP2}

As a consequence of previous results we arrive to the main theorems of this section that completely clarifies the situation concerning the problem $(P_2)$. Recalling that $\omega_{\delta} = \frac{\alpha_{T}}{\beta_{T}} - \frac{\alpha_{0}}{\beta_{0} \delta}$  (see (\ref{Ew})),
we will see that $\rho_T$ and the sign of $\omega_{\delta}$ are the determinant factors in this analysis.

Again, the case $\omega_{\delta} = 0$ is easily solved, because the function to be minimized and the one defining the restriction are proportional.
Hence, the following result can be derived as Proposition \ref{PPROP2}:

\begin{proposition}
Let us assume $d_{min}>0,$ $\rho_T > 1,$ $\lceil \lambda_T \rceil < \lfloor  \rho_T \rfloor$ and $\omega_{\delta} = 0.$
Then any feasible pair $(\overline{N},d)$  with $E_T(\overline{N},d) = \gamma_T$ is a solution to problem $(P_2)$.
In particular, the pairs $(\overline{N}, \overline{d}^{\overline{N}})$ with $\overline{N} \in \{\lceil \lambda_T \rceil, \ldots, \lfloor  \rho_T \rfloor\}$ and $\overline{d}^{\overline{N}} = (\overline{d}_{1},...,\overline{d}_{1}),$ where
\begin{equation}
\overline{d}_{1}=\dfrac{-\alpha_{T} \overline{N} + \sqrt{(\alpha_{T}\overline{N})^{2}+4\beta_{T} \overline{N} \gamma_T}}{2\beta_T \overline{N}}.
\label{E800}
\end{equation}
\label{PPROP3}
\end{proposition}

Let us now continue by studying the more frequent case $\omega_{\delta} > 0$. Here, we have to distinguish two different situations, depending on $\rho_T \in \mathbb{N}$ or not:

\begin{theorem}
Let us assume that $d_{min} > 0,$ $\lceil \lambda_T \rceil < \lfloor  \rho_T \rfloor$ and $\omega_{\delta} > 0$.
\begin{itemize}
\item [a)] If $\rho_T \in \mathbb{N}$, $\rho_T \geq 2,$ then the unique solution to problem $(P_{2})$ is given by $\overline{N} = \rho_T$ and  $\overline{d}^{\overline{N}}=(d_{min},...,d_{min}),$
\item [b)] If $\rho_T \not\in \mathbb{N},$ then the unique solution to problem $(P_{2})$ is given by $\left( \overline{N},\overline{d}^{\overline{N}} \right),$ where:
\begin{itemize}
\item [i)]  $\overline{N} = \lceil \rho_T \rceil$ and  $\overline{d}^{\overline{N}}=(d_{min},...,d_{min}),$  or
\item [ii)]  $\overline{N} = \lfloor \rho_T \rfloor$ and
$\overline{d}^{\overline{N}}=(\overline{d}_{1},...,\overline{d}_{1})$, with $\overline{d}_{1}$ given by (\ref{E800}).
\end{itemize}
\end{itemize}
\label{TPROP21a}
\end{theorem}
The proof is similar to that of Theorem \ref{TPROP2} and postponed to Appendix $2$.

Finally, the case $\omega_{\delta} < 0$ is studied in the next theorem and its proof is detailed in Appendix $3$.

\begin{theorem}
Let us assume $d_{min} > 0,$ $\rho_T > 1,$ $\lceil \lambda_T \rceil < \lfloor  \rho_T \rfloor$ and  $\omega_{\delta} < 0$. Then,
a solution to problem $(P_{2})$ is given by $\left( \overline{N},\overline{d}^{\overline{N}} \right),$ with $\overline{N} = \lceil \lambda_T \rceil$ and
\begin{itemize}
\item [a)]  $\overline{d}^{\overline{N}}=(\underbrace{d_{min},...,d_{min}}_{K},\underbrace{d_{max},...,d_{max}}_{\overline{N}-K}),$
when $K$ defined by (\ref{E720}), with $N = \overline{N},$ belongs to $\mathbb{N} \cup \{0\};$ otherwise,
\item [b)] $\overline{d}^{\overline{N}}=(\underbrace{d_{min},...,d_{min}}_{K},d^{*},\underbrace{d_{max},...,d_{max}}_{\overline{N}-K-1}),$
with $K$ defined by (\ref{E740}) and $d^\star$ satisfies (\ref{E750}), both with $N = \overline{N}.$
\end{itemize}
\label{TPROP21b}
\end{theorem}

\begin{remark}
\begin{itemize}
\item [a)] Once more, let us emphasize that when $\omega_{\delta}>0$ the optimal value of $N$ is the largest one within its range of possibilities (i.e. it is a hyperfractionated type treatment),
while in the case $\omega_{\delta} < 0$ the optimal value is the smallest one (i.e. it is a hypofractionated type treatment).
This classification in terms of $\omega_{\delta}$ was described in \cite{Mizuta2012}, considering nonnegative doses.
\item [b)] For the hypofractionated case, the single exposure is chosen in \cite{Mizuta2012} as the preferred one. But this dose could be too large in practice and then two, three or more fractions would have
to be tried until an acceptable one is found. This fact is remarked in \cite{Bortfeld2015} by saying that the case $\omega_{\delta} \leq 0$
``needs careful consideration since the validity of the model may be limited if $N$ is small and the dose per fraction is large".
Under our approach, this difficulty is overcome and we get the optimal number of dose fractions directly and its value, as in the other case.
\item [c)] The uniqueness of solution fails  when $\omega_{\delta} < 0$ because (as it was said) any permutation of the coordinates of the indicated solution provides a new one.
\end{itemize}
\end{remark}

In the following examples we will show that all the above possibilities mentioned in Theorems \ref{TPROP21a} and \ref{TPROP21b} can appear in practice:

\begin{example}
Let us continue with the same parameters than in  Example \ref{ejemplo1}, which are:
 $\alpha_{0} = 0.04 \ Gy^{-1}$, $\beta_{0}=0.02 \ Gy^{-2},$ $\alpha_{T} = 0.05 \ Gy^{-1},$ $\beta_{T} = 0.005 \ Gy^{-2},$  $ d_{min}=1 \ Gy$ and $d_{max}=6 \ Gy.$

\begin{itemize}
\item [i)] If $\delta=1$, then $\omega_{\delta}=8 >0$.
Hence, we are considering the problem
$$ (P_{21}) \left\{ \begin{array}{ll}
\textrm{Minimize  } E_{OAR}(N,d)  = 0.04 \displaystyle \sum_{i=1}^{N} d_{i} + 0.02 \sum_{i=1}^{N} d_{i}^{2}, \\
\textrm{subject to } N \in \mathbb{N}, d_{i} \in \mathbb{R} \ \textrm{ such that } \\
E_T(N,d)  =  0.05 \displaystyle \sum_{i=1}^{N} d_{i} + 0.005 \sum_{i=1}^{N} d_{i}^{2} \geq \gamma_T, \\
1 \leq d_{i} \leq 6, i=1,...,N.
\end{array}
\right.  $$
When $\gamma_{T} = 4$, we can easily calculate that $\lambda_T \approx 8.33$ and $\rho_{T}  \approx 72.72$.
Applying Theorem \ref{TPROP21a}-b) the unique solution for $ (P_{21}) $ is one of these two pairs:
$N_1=\lceil \rho_T \rceil = 73$, $\overline{d}^{N_1}=(1,...,1)$ with $E_{OAR}(N_1, \overline{d}^{N_1}) = 4.38$  or
$N_2 = \lfloor \rho_T \rfloor = 72$, $\overline{d}^{N_2}=(\overline{d}_{1},...,\overline{d}_{1})$ with $\overline{d}_{1} \approx 1.00926 \ Gy$ and
the objective function value $E_{OAR}(N_2,\overline{d}^{N_2}) \approx 4.373.$ Clearly, we choose the second pair.

\item [ii)]  Previous case is the most frequent in practice when $\omega_{\delta} >0$, but for some specific values of $\gamma_T$ the alternative indicated
in Theorem \ref{TPROP21a}-b) occurs. For instance, taking $\gamma_{T} = 4.014$ in  $(P_{21})$, the values become $\lambda_T \approx 8.36$ and $\rho_{T}  \approx 72.98$.
Now, the two candidates for being the unique solution for $(P_{21}) $ are:
$N_1=\lceil \rho_T \rceil = 73$, $d^{N_1}=(1,...,1)$ with $E_{OAR}(N_1, \overline{d}^{N_1}) = 4.38$, as before, and
$N_2 = \lfloor \rho_T \rfloor = 72$, $\overline{d}^{N_2}=(\overline{d}_{1},...,\overline{d}_{1})$ with $\overline{d}_{1} \approx 1.0125 \ Gy$ and $E_{OAR}(N_2,d^{N_2}) \approx 4.39.$
It is apparent that here the solution is the first one.

\item [iii)] If $\delta=0.1$, then $\omega_{\delta}=-10 < 0$ and the problem under study is
$$ (P_{22}) \left\{ \begin{array}{ll}
\textrm{Minimize  } E_{OAR}(N,d)  = 0.004 \displaystyle \sum_{i=1}^{N} d_{i} + 0.0002 \sum_{i=1}^{N} d_{i}^{2}, \\
\textrm{subject to } N \in \mathbb{N}, d_{i} \in \mathbb{R} \ \textrm{ such that } \\
E_T(N,d)  = 0.05 \displaystyle \sum_{i=1}^{N} d_{i} + 0.005 \sum_{i=1}^{N} d_{i}^{2} \geq \gamma_T, \\
1 \leq d_{i} \leq 6, i=1,...,N.
\end{array}
\right.  $$

For $\gamma_T = 4.35$, we have that $\lambda_T \approx 9.06$ and $\rho_{T} \approx 79.09$. Therefore, applying Theorem \ref{TPROP21b}-b), we know that a solution for $(P_{22})$ is given by $(\overline{N},\overline{d})$ with $\overline{N}=\lceil \lambda_T \rceil = 10$ and $\overline{d}=(1,d^{*},\underbrace{6,...,6}_{8})$ with $d^{*} \approx 5.77 \ Gy$. In this case, $E_{OAR}(10,\overline{d}) \approx 0.2835$. As in Example \ref{ejemplo1}, for comparison reasons, we can calculate that the solution of solution of $(P_{22}^{11})$ is $\tilde{d}_1 = (1,1, d_1^{*},\underbrace{6,...,6}_{8})$ with $d_1^{*} \approx 5.247 \ Gy$ producing $E_{OAR}(11,\tilde{d}_1) \approx 0.2845$, that is slightly bigger than $E_{OAR}(10,\overline{d}) $, as expected.

\item [iv)] Again, it can be shown that, for some specific values of the parameter $\gamma_{T}$, the alternative exposed in Theorem \ref{TPROP21b}-a) is true.
In particular, for $\gamma_T = 4.375$, we have $\lambda_T \approx 9.11$ and $\rho_{T} \approx 79.54$ and a solution for $(P_{22})$ is given by $(\overline{N},\overline{d})$ with $\overline{N}=\lceil \lambda_T \rceil = 10$ and $\overline{d}=(1,\underbrace{6,...,6}_{9})$.
\end{itemize}
\label{EX6}
\end{example}

As we did for problem $(P_1)$, here we can deduce quite easily how the solution depends with respect to the parameters defining the problem $(P_2)$.

\begin{corollary}
Let us assume $d_{min}>0,$ $\rho_{0} \geq 2$ and $(\overline{N}, \overline{d}^{\overline{N}})$ is a solution of $(P_2)$. Then,
$\overline{N}$ is an increasing function of $\gamma_T$, decreasing with respect to $\alpha_T$ and $\beta_T$ and independent of $\alpha_0, \beta_0$ and $\delta$. Moreover,
\begin{itemize}
\item [a)] When $\omega_{\delta} > 0,$ $\overline{N}$ is also decreasing with respect to $d_{min}$ and independent of $d_{max}$.
\item [b)] When $\omega_{\delta} < 0,$ $\overline{N}$ is also decreasing with respect to $d_{max}$ and independent of $d_{min}$.
\end{itemize}
\label{CO18}
\end{corollary}

\begin{proof}
It follows from the expressions $\overline{N} = \lfloor \rho_T \rfloor$ or $\lceil \rho_T \rceil$ with $\rho_T$ given by (\ref{rhoT}), when $\omega_{\delta} > 0,$
and $\overline{N} = \lceil \lambda_T \rceil$  with $\lambda_T$ given by (\ref{lamT}), when $\omega_{\delta} < 0.$
\end{proof}

For the reader's convenience, we have summarized the complete algorithm for the resolution of the problem $(P_2)$ in Table \ref{solP2}.

\begin{table}
\begin{center}
\begin{tabular}{| c | }
\hline
ALGORITHM FOR SOLVING $(P_2)$ \\
\hline
\hline
DATA: $\alpha_{T}, \beta_{T}, \alpha_{0}, \beta_{0}, d_{min},
d_{max}, \delta$ and $\gamma_{T}$ \\
(all positive, $d_{min} < d_{max}$ and $\delta \leq 1$) \\
\hline
CALCULATE:  $\omega_{\delta} = \dfrac{\alpha_T}{\beta_T} - \dfrac{\alpha_0}{\beta_0 \delta},$
$\lambda_{T} = max\left\lbrace  1, \dfrac{\gamma_{T}}{\varphi_{T}(d_{max})} \right\rbrace$ and \\
$\rho_{T} = max\left\lbrace  1, \dfrac{\gamma_{T}}{\varphi_{T}(d_{min})} \right\rbrace,$
with $\varphi_{T}(r)=\alpha_{T} r + \beta_{T} r^{2}.$ \\
\hline
IF $\rho_T = 1$, the pair $(\overline{N}, \overline{d}^{\overline{N}})  = (1,d_{min})$ is the UNIQUE SOLUTION of $(P_{2})$.\\
\hline
IF $\rho_T > 1$ and $\lceil \lambda_T \rceil = \lfloor  \rho_T \rfloor$, the UNIQUE SOLUTION of $(P_2)$ \\ is the pair
$(\overline{N}, \overline{d}^{\overline{N}})$ with $\overline{N} = \lceil \rho_T \rceil$  and $\overline{d}^{\overline{N}} = (d_{min},...,d_{min}).$ \\
\hline
IF  $\rho_T \in \mathbb{N}$, $\rho_T \geq 2$, $\lceil \lambda_T \rceil < \lfloor  \rho_T \rfloor$ and $\omega_{\delta} > 0,$  the pair $(\overline{N}, \overline{d}^{\overline{N}})$ \\
with  $\overline{N} = \rho_T$ and $\overline{d}^{\overline{N}} = (d_{min},...,d_{min})$ is the UNIQUE SOLUTION of $(P_{2})$.\\
\hline
IF  $\rho_T \not\in \mathbb{N},$ $\lceil \lambda_T \rceil < \lfloor  \rho_T \rfloor$ and $\omega_{\delta} > 0,$ \\ take $(\overline{N}_1, \overline{d}_1^{\overline{N}_1}),$
with  $\overline{N}_1 =\lfloor \rho_T \rfloor$ and $\overline{d}_1^{\overline{N}_1} = (\overline{d}_{1},...,\overline{d}_{1})$
\\  where $\overline{d}_{1}=\dfrac{-\alpha_{T} \overline{N}_1 + \sqrt{(\alpha_{T}\overline{N}_1)^{2}+4\beta_{T} \overline{N}_1 \gamma_{T}}}{2\beta_{T} \overline{N}_1}$.
\\
Also take $\overline{N}_2 = \lceil \rho_T \rceil$  and $\overline{d}_2^{\overline{N}_2} = (d_{min},...,d_{min})$. \\
CALCULATE: $E_{OAR}( \overline{N}_1,\overline{d}_1^{\overline{N}_1})$  and $E_{OAR}( \overline{N}_2,\overline{d}_2^{\overline{N}_2}).$ \\
A SOLUTION of $(P_2)$ is the pair $(\overline{N}, \overline{d}^{\overline{N}})$ that minimizes $E_{OAR}$ between them.\\
\hline
IF $\rho_T > 1,$ $\lceil \lambda_T \rceil < \lfloor  \rho_T \rfloor$ and $\omega_{\delta} < 0,$ \\ take $\overline{N }= \lceil \lambda_T \rceil$ and CALCULATE $M = \dfrac{\overline{N} \varphi_{T}(d_{max})-\gamma_{T}}{\varphi_{T}(d_{max})-\varphi_{T}(d_{min})}$. \\
IF $M  \in \mathbb{N} \cup \{0\}$, take $K = M$ and $\overline{d}^{\overline{N}} = (\underbrace{d_{min},...,d_{min}}_{K},\underbrace{d_{max},...,d_{max}}_{\overline{N}-K})$. \\
IF $M  \not\in \mathbb{N} \cup \{0\}$, take $K = \lfloor M \rfloor$ and $\overline{d}^{\overline{N}} = (\underbrace{d_{min},...,d_{min}}_{K},d^{*},\underbrace{d_{max},...,d_{max}}_{\overline{N}-K-1}),$ \\ with $d^{*} > 0$ and $\varphi_{T}(d^{*})= \gamma_{T} - K \varphi_{T}(d_{min})-(\overline{N}-K-1) \varphi_{T}(d_{max})$. \\
A SOLUTION of $(P_2)$ is the pair $(\overline{N}, \overline{d}^{\overline{N}}).$\\
\hline
IF $\rho_T  > 1$, $\lceil \lambda_T \rceil < \lfloor  \rho_T \rfloor$ and $\omega_{\delta} = 0,$ ANY FEASIBLE PAIR $(\overline{N},d)$  such that \\
$E_T(\overline{N},d) = \gamma_T$ is a SOLUTION for $(P_2)$. \\
In particular, the pairs $(\overline{N}, \overline{d}^{\overline{N}})$ with $\overline{N} \in \{\lceil \lambda_T \rceil, \ldots, \lfloor  \rho_T \rfloor\}$ and $\overline{d}^{\overline{N}} = (\overline{d}_{1},...,\overline{d}_{1}),$
\\  where $\overline{d}_{1}=\dfrac{-\alpha_{T} \overline{N} + \sqrt{(\alpha_{T}\overline{N})^{2}+4\beta_{T} \overline{N} \gamma_T}}{2\beta_T \overline{N}}.$\\
\hline
\end{tabular}
\caption{Complete solution for problem $(P_2)$ in algorithmic form.}\label{solP2}
\end{center}
\end{table}

 \subsection{Equivalent treatments}
 \label{EQTR}

 We will finish this section by mentioning another application of previous results. Let us begin by introducing biologically equivalent treatments.
 Two treatments of radiotherapy with doses $d_{1},...,d_{N}$ and $\tilde{d}_{1},...,\tilde{d}_{\tilde{N}}$ are said to be biologically equivalent for a certain tumor with characteristic parameters $\alpha_{T}$ and $\beta_{T}$ when they have the same effect, that is,
$$ \displaystyle \alpha_{T} \sum_{i=1}^{N} d_{i}+\beta_{T} \sum_{i=1}^{N} d_{i}^{2}=\displaystyle \alpha_{T} \sum_{i=1}^{\tilde{N}} \tilde{d}_{i}+\beta_{T} \sum_{i=1}^{\tilde{N}} \tilde{d}_{i}^{2}. $$
When all the doses are equal for both treatments, previous concept leads to biologically equivalent doses (BED) (see \cite{Bortfeld2015}), that can be calculated very easily (see \cite{ref8} for instance) from the equality $$ N(\alpha_T d + \beta_T d^2) =
\tilde{N}(\alpha_{T}\tilde{d}+\beta_{T} \tilde{d}^2).$$

A clearly interesting question is to determine (among all the equivalent treatments) which one uses the lowest total dose. As in this case we are not paying attention to the effect of radiation on the OAR, this can be formulated in mathematical terms as the following optimization problem:
$$ (P_3)\left\{
\begin{array}{ll}
\begin{aligned}
& \textrm{Minimize } \sum_{i=1}^{N} d_{i}, \\
& \textrm{subject to } N \in \mathbb{N}, d \in \mathbb{R}^{N} \textrm{ such that  } \\
& \alpha_{T} \sum_{i=1}^{N} d_{i} + \beta_{T} \sum_{i=1}^{N} d_{i}^{2} = \gamma_T, \\
& d_{min}\leq d_{i}\leq d_{max} \textrm{,  } i=1,...,N.
\end{aligned}
\end{array}
\right. $$

After our study, we deduce that the solutions to this problem are treatments of type (\ref{E710}) or (\ref{E730}), see Theorem \ref{teo22}-ii).
Let us mention that these are the same than in the case $\omega_{\delta} < 0$, see Theorem \ref{TPROP21b}, and consequently the hypofractionated protocols
are also optimal in this sense.

\section{Conclusions}

In this work, we have derived the analytical expressions for the optimal total number of radiations $N$ and their specific doses $d$ for problems $(P_1)$ and  $(P_2)$.
They are presented  in Tables \ref{solP1} and \ref{solP2} in algorithmic form (see also Theorems \ref{TPROP2-}, \ref{TPROP2}, \ref{TPROP3}, \ref{TPROP21a} and \ref{TPROP21b})
and there exists a clear parallelism between the structure of the solutions for both problems. We have proved that they essentially depend on the sign of the quantity
$$ \omega_{\delta} = \dfrac{\alpha_T}{\beta_T} - \dfrac{\alpha_0}{\beta_0 \delta}.$$
For fixed $N$, this fact is well known in the literature and it has been reported several times in different frameworks (see for instance \cite{Mizuta2012}, \cite{Bertuzzi2013} and \cite{Bruni2019}).
Moreover, this is consistent with some clinical findings as noted in \cite{Mizuta2012}.

When $\omega_{\delta} > 0$, we have shown that the optimal number of doses $N$ are $\lfloor \rho_0 \rfloor$ for $(P_1)$ and $\lfloor \rho_T \rfloor$ or $\lceil \rho_T \rceil$ for $(P_2)$,
the upper values of their ranges of interest (i.e. hyperfractionated type treatments) with equal doses;
while in the case $\omega_{\delta} < 0,$ the optimal values of $N$ are $\lfloor \lambda_0 \rfloor$ or $\lceil \lambda_0 \rceil$ for $(P_1)$ and $\lfloor \lambda_T \rfloor$ for $(P_2)$,
the lower values of those ranges (i.e. hypofractionated type treatments). In this last case, let us stress that not all doses
have to be maximum; in fact, some of them may be minimum and at most one of them can take an intermediate value. For non-uniform protocols, the lack of uniqueness for the solution can be used to our benefit, because the doses can be administered in any order depending on various external factors such as the condition of the patient. The study concerning the derivation of the optimal number of doses $N$ had already been performed for example in \cite{Jones1995} in the hyperfractionated case, but (as far as we know) it is completely new for the hypofractionated case.

Let us emphasize again that the calculations to apply all these results are elementary and can be carried out using a pocket calculator from the initial data.
Of course, the algorithms described in Tables \ref{solP1} and \ref{solP2} can be implemented quite straightforwardly in any platform using any programming language to make them more accessible.

We hope that these theoretical results may provide useful insights to address more complete models (including repopulation terms and multiple OAR)
and that, ultimately, will lead to some improvement (however small) in clinical practice, due to the impact it would have on the large number of patients who could benefit.

\section*{Acknowledgement}
The authors would like to express their gratitude to Prof. Cecilia Pola (University of Cantabria) for fruitful discussions and helpful comments.

\section*{Appendix $1$: Proof of Theorem \ref{TPROP3}}

The expression given in $i)$ is derived exactly as in Theorem \ref{TPROP2-} for the values $N \leq \lambda_0$. Taking into account that $d^{*}$ can be very close to $d_{max}$ or $d_{min}$,
item $ii)$ can be seen as a kind of special case of $iii)$. So, we will focus on proving $iii)$ that it is the most complicated case. To that end, it is enough to show that if $(\underbrace{d_{min},...,d_{min}}_{K},d^{*},\underbrace{d_{max},...,d_{max}}_{N-K-1})$ is a solution for $(P_{1}^N)$ and $(\underbrace{d_{min},...,d_{min}}_{\tilde{K}},\tilde{d}^{*},\underbrace{d_{max},...,d_{max}}_{N-\tilde{K}})$ is a solution for $(P_{1}^{N+1})$, with $N > \lambda_{0}$, then the following relation holds
\begin{equation}
K d_{min}+ d^{*}+(N-K-1)d_{max} \leq \tilde{K} d_{min}+ \tilde{d}^{*}+(N-\tilde{K})d_{max}.
\label{E600}
\end{equation}
Together with (\ref{eq1}) and the assumption $\omega_{\delta} < 0$, this implies $iii)$, because (\ref{E600}) means that the values of the objective function $E_T$ at the solutions are decreasing with $N$ and therefore, the maximum value will be attained at $\overline{N} = \lceil \lambda_0 \rceil$, the lowest value of $N$ in the set $(\lambda_0,\rho_0] \cap \mathbb{N}$.

Comparing their expressions in the form (\ref{K2}) with $N+1$ and $N$, resp., we conclude that $\tilde{K} \geq K+1$. Hence, if we denote $K_0 = \tilde{K} - K \in \mathbb{N}$, the inequality (\ref{E600}) can be written as
\begin{equation}
d^{*}-\tilde{d}^{*} \leq K_0 d_{min}+ (1-K_0)d_{max}.
\label{E610}
\end{equation}
Let us recall that $d^\star$ satisfies (\ref{E400}) and $\tilde{d}^{*}$ verifies
\begin{equation}
\varphi_{0}(\tilde{d}^{*})= \gamma_{OAR} - \tilde{K} \varphi_{0}(d_{min})-(N-\tilde{K}) \varphi_{0}(d_{max}).
\label{E620}
\end{equation}

We will show that (\ref{E610}) holds dividing the argumentation in three cases:

\underline{Case 1.- Suppose that $K_0 d_{min}+ (1-K_0)d_{max} \leq 0$}. We choose the point
$$(d_1,...,d_N) = (\underbrace{\dfrac{K_0}{K_0-1}d_{min},...,\dfrac{K_0}{K_0-1}d_{min}}_{K_0-1},\underbrace{d_{min},...,d_{min}}_{\tilde{K}-K_0},\tilde{d}^{*},\underbrace{d_{max},...,d_{max}}_{N-\tilde{K}}),$$ that under the assumption satisfies  the bounds restrictions and
$$E_{OAR}(N,d) = \alpha_0 \delta \left( \tilde{K}d_{min}+ \tilde{d}^{*} + (N-\tilde{K}) d_{max}\right) + $$
$$ + \beta_0 \delta^2 \left(\dfrac{(K_0)^2}{K_0-1}d_{min}^2+ (\tilde{K}-K_0) d_{min}^2 + (\tilde{d}^{*})^2 + (N-\tilde{K}) d_{max}^2  \right) \geq $$
$$  \geq \tilde{K}\varphi_0(d_{min}) + \varphi_0(\tilde{d}^{*}) + (N-\tilde{K}) \varphi_0(d_{max})  = \gamma_{OAR}.$$
This means that it is feasible for the problem $(\ref{PEQ-})$. Taking into account  Remark \ref{R2}, we get (\ref{E610}).

\underline{Case 2.- Suppose now that $K_0 d_{min}+ (1-K_0)d_{max} > 0$} and furthermore

\underline{$K_0 \varphi_0(d_{min})+ (1-K_0)\varphi_0(d_{max}) \leq 0$}. We can argue similarly choosing
\begin{equation}
(d_1,...,d_N) = (\underbrace{d_{min},...,d_{min}}_{K},\tilde{d}^{*},\underbrace{d_{max},...,d_{max}}_{N-K-1}).
\label{E625}
\end{equation}
Due to (\ref{E620}) and the hypothesis we have
$$E_{OAR}(N,d) = K\varphi_0(d_{min}) + \varphi_0(\tilde{d}^{*}) + (N-K-1) \varphi_0(d_{max}) = $$
$$ = \gamma_{OAR} - K_0\varphi_0(d_{min}) + (K_0-1) \varphi_0(d_{max})  \geq \gamma_{OAR}.$$
Again, we have a feasible point for the problem $(\ref{PEQ-})$ and therefore we deduce  $d^{*} \leq \tilde{d}^{*}$ and hence (\ref{E610}), because
$$ d^{*} - \tilde{d}^{*} \leq 0 < K_0 d_{min}+ (1-K_0)d_{max}.$$

\underline{Case 3.- Finally, suppose that $K_0 d_{min}+ (1-K_0)d_{max} > 0$} and moreover

\underline{$K_0 \varphi_0(d_{min})+ (1-K_0)\varphi_0(d_{max}) > 0$}.
Here, we introduce the auxiliary function defined for $s \in [0,1]$ by

\small
\[ G(s) = \sqrt{\alpha_0^2 + 4\beta_0 \left(\gamma_{OAR}-(\tilde{K} - s K_0)\varphi_{0}(d_{min}) -(N-\tilde{K}+s(K_0-1))\varphi_{0}(d_{max})\right)}. \]
\normalsize
Solving the quadratic equations (\ref{E400})  and (\ref{E620}), it is easy to derive that
\[ d^* = \dfrac{-\alpha_0 + G(1)}{2\beta_0 \delta},  \ \ \tilde{d}^* = \dfrac{-\alpha_0 + G(0)}{2\beta_0 \delta}.\]
Using the Mean Value Theorem, we deduce that there exists $\theta \in (0,1) $ such that
\[ d^* -\tilde{d}^* = \dfrac{G(1)-G(0)}{2\beta_0 \delta} =  \dfrac{G'(\theta)}{2\beta_0 \delta} = \dfrac{K_0 \varphi_0(d_{min})+ (1-K_0)\varphi_0(d_{max})}{G(\theta) \delta}.\]
Therefore, the inequality (\ref{E610}) is equivalent to
\begin{equation}
\dfrac{K_0 \varphi_0(d_{min})+ (1-K_0)\varphi_0(d_{max})}{K_0 d_{min}+ (1-K_0)d_{max}} \leq G(\theta) \delta.
\label{E630}
\end{equation}
Under the present hypotheses, the function $G$ is strictly increasing and, since $\theta$ is an unknown value in $(0,1)$, we will verify that (\ref{E630}) is valid if  it holds for $\theta = 0$.
On the other hand, the value $K_0 \in \mathbb{N}$ is also unknown, but we can verify that the function
\[  F(m) = \dfrac{m \varphi_0(d_{min})+ (1-m)\varphi_0(d_{max})}{m d_{min}+ (1-m)d_{max}} = \]
\[ = \alpha_0 \delta + \beta_0 \delta^2\left( \dfrac{m d_{min}^2+ (1-m)d_{max}^2}{m d_{min}+ (1-m)d_{max}}\right) ,\]
is strictly decreasing, because
\[  F'(m) = \beta_0 \delta^2\dfrac{d_{min}d_{max}(d_{min}-d_{max})}{(m d_{min}+ (1-m)d_{max})^2}< 0.\]
Hence, the inequality (\ref{E630}) will be true if $F(1) \leq G(0)\delta.$ We conclude by noting that
\[ \tilde{d}^* \in [d_{min},d_{max}] \Longleftrightarrow \varphi_0(\tilde{d}^*) \in [\varphi_0(d_{min}),\varphi_0(d_{max})]  \Longleftrightarrow  \]
\[ \stackrel{(\ref{E620})}{\Longleftrightarrow} \gamma_{OAR} - \tilde{K} \varphi_{0}(d_{min})-(N-\tilde{K}) \varphi_{0}(d_{max})  \in [\varphi_0(d_{min}),\varphi_0(d_{max})].  \]
Then,
\[ G(0) = \sqrt{\alpha_0^2 + 4\beta_0 \left(\gamma_{OAR}- \tilde{K} \varphi_{0}(d_{min}) -(N-\tilde{K})\varphi_{0}(d_{max})\right)} \geq  \]
\[ \geq \sqrt{\alpha_0^2 + 4\beta_0 \varphi_{0}(d_{min}}) \geq \alpha_0 + \beta_0 \delta d_{min}= \dfrac{F(1)}{\delta}, \]
as asserted.

\section*{Appendix $2$: Proof of Theorem \ref{TPROP21a}}

It follows the same lines to that of Theorem \ref{TPROP2}.

\underline{Case a).- Assume $\rho_T \in \mathbb{N}$, $\rho_T \geq 2.$}

As usual, we divide the interval for feasible values of $N$ in two parts: $[\lambda_T,\rho_T) \cap \mathbb{N}$ and $[\rho_T,+\infty)\cap \mathbb{N}.$

In order to study the dependence with respect to $N$ in the interval $[\lambda_T,\rho_T)$, thanks to Proposition \ref{PROP2} (with $\omega_{\delta} > 0$) and (\ref{E700}), it is enough to consider the auxiliary function
$$\psi_1(N)= N \overline{d}_{1}=\dfrac{-\alpha_{T} N + \sqrt{(\alpha_{T} N)^{2}+4\beta_T N \gamma_T}}{2\beta_T}.$$
Once more, it follows easily that $\psi_1(N)$ is an strictly increasing function. Since we are assuming $\rho_T \in \mathbb{N}$ and
$\rho_T \geq 2$, then $\psi_1$ will take its maximum value in the set $[\lambda_T,\rho_T) \cap \mathbb{N}$ at $N_1 = \rho_T - 1$.
Therefore, the candidate for solution to problem $(P_2)$ is given by the pair $(N_1,\overline{d}^{N_1})$ with $\overline{d}^{N_1}= (\overline{d}_{1},...,\overline{d}_{1}),$ where $\overline{d}_{1}$ is given by
\begin{equation}
\overline{d}_1= \dfrac{-\alpha_T N_1 + \sqrt{(\alpha_T N_1)^{2}+4\beta_T N_1 \gamma_T}}{2\beta_{T} N_1}.
\label{E7510}
\end{equation}

On the other hand, in the interval $[\rho_T,+\infty)$, we know that the other candidate for solution to problem $(P_2)$ is given by the pair $(N_2,\overline{d}^{N_2})$ with $N_2 = \rho_T$ and
$\overline{d}^{N_2} =(d_{min},...,d_{min}).$

To derive that $(N_2,\overline{d}^{N_2})$ is the unique solution to problem $(P_{2})$, we will show that
\begin{equation}
E_{OAR}(N_2,\overline{d}^{N_2}) < E_{OAR}(N_1,\overline{d}^{N_1}).
\label{E7530}
\end{equation}

Following the same idea to that of the proof of Theorem \ref{TPROP2}, we introduce the auxiliary function
\[ H_1(x) = N_1(x\overline{d}_1+\overline{d}_1^2)-N_2(x d_{min}+d_{min}^2), \ \ \ x \in [\dfrac{\alpha_0}{\beta_0 \delta}, +\infty).\]
Taking into account (\ref{E7510}) and that $N_2 \varphi_T(d_{min}) = \gamma_T$ (by the definition of $\rho_T$), it can be checked that $H_1'(x) = N_1\overline{d}_1-N_2 d_{min} < 0,$
since $N_1 < N_2.$

Using that also $\gamma_T = N_1\varphi_T(\overline{d}_1)$, we get that $H_1\left(\dfrac{\alpha_T}{\beta_T}\right) = 0$
and from the assumption $\omega_{\delta} >0$ (see (\ref{Ew})), it follows that $H_1\left(\dfrac{\alpha_0}{{\beta_0 \delta}}\right) > 0$, which is equivalent to (\ref{E7530}).

\underline{Case b).- Assume $\rho_T \not\in \mathbb{N}$}.
Here, the optimal value of $N$ in the interval $[\lambda_T,\rho_T)$ is  $N_1 = \lfloor  \rho_T \rfloor$ and $\overline{d}^{N_1}= (\overline{d}_{1},...,\overline{d}_{1})$ with $\overline{d}_{1}$ given by (\ref{E7510}). In the interval $[\rho_T,+\infty)$, the other candidate is $N_2 = \lceil \rho_T \rceil$ with $\overline{d}^{N_2} =(d_{min},...,d_{min}).$

When $\rho_T \not\in \mathbb{N},$ any of them can provide the unique solution to problem $(P_{2})$ (see for instance Example \ref{EX6}).

\section*{Appendix $3$: Proof of Theorem \ref{TPROP21b}}

When $\omega_{\delta} < 0$, it is still true that $N_2 = \lceil \rho_T \rceil$ and $\overline{d}^{N_2} =(d_{min},...,d_{min}).$
Arguing as in Appendix $1$, the candidate when $N$ runs $[\lambda_T,\rho_T) \cap \mathbb{N}$ is
$N_1 = \lceil \lambda_T \rceil$ with $\overline{d}^{N_1}$ given by Theorem \ref{TPROP21b}-$a)$ or $b)$ and $\overline{N} = N_1$,
thanks to Theorem \ref{teo22}$-ii)$. We will conclude by showing that
\begin{equation}
 E_{OAR}(N_1,\overline{d}^{N_1}) \leq E_{OAR}(N_2,\overline{d}^{N_2}).
\label{E815}
\end{equation}
Let us argue with the expression $b)$ for $\overline{d}^{N_1}$, because (as we have pointed out before) the value $d^*$ can be very close to $d_{min}$ or $d_{max}$
and hence item $a)$ can be seen as a special case of $b)$.
Therefore, the inequality (\ref{E815}) is equivalent to
\begin{equation}
K\varphi_0(d_{min})+\varphi_0(d^\star)+(N_1-K-1)\varphi_0(d_{max}) \leq N_2\varphi_0(d_{min}).
\label{E820}
\end{equation}

For proving (\ref{E820}), we consider again a linear function such as
$$ H_2(x) = (N_2-K)(x d_{min}+d_{min}^2) -(xd^* +(d^*)^2)-(N_1-K-1)(x d_{max}+d_{max}^2).$$
By construction, we know that
$$K\varphi_T(d_{min})+\varphi_T(d^\star)+(N_1-K-1)\varphi_T(d_{max}) = \gamma_T \leq N_2\varphi_T(d_{min}). $$
This is equivalent to say that $H_2\left(\dfrac{\alpha_T}{\beta_T}\right) \geq 0$.

If $H_2$ is an increasing function, since $\omega_{\delta} < 0$, we will have
$$H_2\left(\dfrac{\alpha_0}{\beta_0 \delta}\right) \geq H_2\left(\dfrac{\alpha_T}{\beta_T}\right) \geq 0,$$ which gives (\ref{E820}).

So, taking into account that
$$ H'_2(x) = (N_2-K)d_{min}- d^*-(N_1-K-1)d_{max},$$
let us finish the proof by showing that $H'_2(x) \geq 0$.

If $N_2d_{min} > N_1d_{max},$ this is true straightforwardly, because we know that $N_1d_{max} > K d_{min}+d^*+(N_1-K-1)d_{max}.$

When $N_2d_{min} \leq N_1d_{max},$ we can argue as in the proof of Theorem \ref{TPROP3}, taking the point
$$\tilde{d}= (d_1,...,d_{N_1}) = (\dfrac{N_2}{N_1}d_{min},...,\dfrac{N_2}{N_1}d_{min}),$$  that satisfies  the bounds restrictions and
$$E_{T}(N_1,\tilde{d}) = \alpha_T N_2 d_{min}+\beta_T \dfrac{(N_2)^2}{N_1}d_{min}^2 \geq N_2\varphi_T(d_{min}) \geq \gamma_{T}.$$
This means that it is feasible for the problem $(P_2^{N_1,-})$. Taking into account  that $(N_1,\overline{d}^{N_1})$ is a
solution for that problem, see Proposition \ref{PROP2}-$ii)$ and Remark \ref{R2}, we get $H'_2(x) \geq 0$.

\end{document}